\documentclass[english,twoside]{article}

\title{
Large-time asymptotic stability of Riemann shocks\\
of scalar balance laws}
\author{V. Duch\^{e}ne\thanks{Univ Rennes, CNRS, IRMAR - UMR 6625, F-35000 Rennes, France. \url{vincent.duchene@univ-rennes1.fr}} \and L.M. Rodrigues\thanks{Univ Rennes \& IUF, CNRS, IRMAR - UMR 6625, F-35000 Rennes, France. \url{luis-miguel.rodrigues@univ-rennes1.fr}.
\hspace*{2em}Research of LMR was partially supported by the city of Rennes.
}}
\date{\today}

\usepackage{babel}        
\usepackage[utf8]{inputenc} 
\usepackage[T1]{fontenc}   
\usepackage{amsmath, amsthm,amsfonts,amssymb}
\usepackage{float,subfigure}
\usepackage[pdftex]{graphicx}    
\usepackage{fancyhdr}        
\usepackage{url,hyperref}    
\usepackage{mathtools} 
\usepackage{titling}  

\makeatletter
\let\Title\@title
\let\Author\@author
\makeatother
 
\fancypagestyle{mystyle}{%
	\fancyfoot{}
	\fancyhead[CO]{\Title} 
	\fancyhead[CE]{\scriptsize V. Duch\^{e}ne and L.M. Rodrigues} 
	\fancyhead[RE]{\scriptsize\today}  
	\fancyhead[LO]{}  
	\fancyhead[LE,RO]{\scriptsize\thepage}
	}
\fancypagestyle{plain}{\fancyhead{} } 

\newtheorem{Theorem}{Theorem}[section]
\newtheorem{Proposition}[Theorem]{Proposition}
\newtheorem{Corollary}[Theorem]{Corollary}
\newtheorem{Lemma}[Theorem]{Lemma}
\newtheorem{Remark}[Theorem]{Remark}

\newcommand\uU{{\underline U}}
\newcommand\uu{{\underline u}}
\newcommand\cS{{\mathcal S}}
\newcommand\cL{{\mathcal L}}
\newcommand\cC{{\mathcal C}}
\newcommand\ta{{\widetilde a}}
\newcommand\tb{{\widetilde b}}
\newcommand\tr{{\widetilde r}}
\newcommand\tu{{\widetilde u}}
\newcommand\tv{{\widetilde v}}
\newcommand\tpsi{{\widetilde \psi}}
\newcommand{\RR}{\mathbb{R}}
\newcommand{\CC}{\mathbb{C}}
\newcommand{\NN}{\mathbb{N}}

\DeclareMathOperator\dd{d}
\renewcommand\d{\partial}
\DeclareMathOperator\sign{sgn}
\DeclareMathOperator\argmin{arg\,min}
\newcommand{\ie}{{\em i.e.}~}
\newcommand{\id}[1]{\left\vert_{_{#1}}\right.}
\newcommand{\eqdef}{\stackrel{\rm def}{=}}
\DeclarePairedDelimiter\abs{\lvert}{\rvert}
\DeclarePairedDelimiter\Norm{\big\lVert}{\big\rVert}

\DeclareMathOperator{\Div}{div}
\newcommand{\fpar}{f_{\shortparallel}}
\newcommand{\fperp}{f_{\perp}}

\begin{document}
\thispagestyle{empty}
\maketitle

\begin{abstract}
We prove the large-time asymptotic orbital stability of strictly entropic Riemann shock solutions of first order scalar hyperbolic balance laws, under piecewise regular perturbations provided that the source term is dissipative about endstates of the shock. Moreover the convergence towards a shifted reference state is exponential with a rate predicted by the linearized equations about constant endstates. 

\vspace{1em}

\vspace{1em}
\end{abstract}

\section*{Introduction}

In the present contribution, we study the large-time asymptotic behavior of solutions to first order scalar hyperbolic balance laws, that is, in the one-dimensional case\footnote{While most of this work is restricted to dimension one, we also discuss in Section~\ref{s:multiD} how ideas therein extend to multidimensional settings.}, of the form
\begin{equation}\label{eq-u}
\d_t u+\d_x \big(f(u)\big)=g(u), \qquad u:\RR^+\times\RR\to\RR\,,
\end{equation}
in a neighborhood of strictly entropy-admissible Riemann shocks, that is, about strictly entropy-admissible traveling waves with profiles piecewise constant and exhibiting a single discontinuity.

Equations such as~\eqref{eq-u} are prototypes for dynamics where only convective and reaction effects are relevant, and, as such, are ubiquitous in applications, at least as first-order approximations in some particular regimes. In particular, when $f(u)=c(u)\,u$ and $g(u)=r(u)\,u$, it describes the evolution of a density $u$ of point particles moving with speed $c$ and reacting at rate $r$.

In comparison with the purely conservative case encoded by homogeneous conservation laws (\ie with $g\equiv 0$), the (local) well-posedness of the standard initial-value (Cauchy) problem for~\eqref{eq-u} is not significantly altered by the addition of sufficiently smooth (say locally Lipschitz) reaction terms $g$. In particular the theory of  Kru\v zkov~\cite{Kruzhkov} applies and there exists a unique bounded local-in-time entropy weak solution for any bounded initial data. However in contrast the large-time asymptotic behavior of the solutions is expected to be deeply impacted by the presence of the source term, even when it does not lead to reaction blow-up, for instance when $g\in  W^{1,\infty}(\RR)$ or $g$ is dissipative at infinity. This expectation is consistent with the simple observations that the purely reactive case (with $f\equiv0$) assigns a distinguished role to stable zeros of $g$ --- that is, the $u_\star$ such that $g(u_\star)=0$ and $g'(u_\star)<0$ --- and that related growth and decay mechanisms are generically exponentially fast, hence much stronger than the algebraic decay involved in the purely conservative large-time dynamics. That the reaction term plays a dominant role --- at least near equilibria --- is also supported by the fact that the linearized operator about zeros of $g$, that is
\[L=-f'(u_\star)\d_x +g'(u_\star),\] on, say, $BUC^0(\RR)$ with domain $BUC^1(\RR)$ when $f'(u_\star)\neq 0$, $BUC^0(\RR)$ otherwise, is closed densely defined with spectrum $g'(u_\star)+i\RR$ if $f'(u_\star)\neq 0$, $\{g'(u_\star)\}$ otherwise. In other words, the spectral stability of zeros of $g$ as equilibria of~\eqref{eq-u} agrees with their stability as equilibria of the purely reactive equation. Incidentally note that we have used notation $BUC^k$ to denote the set of $\cC^k$ functions whose derivatives up to order $k$ are bounded and uniformly continuous.

Another good grasp at new large-time phenomena (compared to the conservative case) is already obtained from the analysis of the structure of relative equilibria, namely in the present case of traveling waves. Since the presence of a source term discards self-similarity, these are the most natural candidates to serve as asymptotic profiles or building blocks of a large-time description. Traveling waves of~\eqref{eq-u} are given as $u(t,x)=\uU(x-\sigma t)$ with wavespeed $\sigma$ and waveprofile $\uU$ solving
\[
(f(\uU)-\sigma\,\uU)'=g(\uU)\,.
\]
One striking novelty in the non-homogeneous setting is the existence of traveling wave solutions with non-trivial profiles, whereas in the conservative case only piecewise constant profiles are available and the only spatially periodic entropy-admissible profiles are constant. The most obvious ones are obtained by picking two consecutive non degenerate zeros $u_-$ and $u_+$ of $g$, a speed $\sigma\notin f'([u_-,u_+])$ and solving $\uU'=g(\uU)/(f'(\uU)-\sigma)$ between these two zeros. Yet, in this configuration one of the two endstates is spectrally unstable and the corresponding front inherits this instability. More interesting waves are obtained if one allows the presence of a sonic, or characteristic, point in the profile, that is, a point where $f'(\uU)-\sigma$ vanishes. Necessarily then the wavespeed $\sigma$ must be equal to the sound speed $f'(u_\star)$ at a zero $u_\star$ of $g$. In the non degenerate bistable case when $u_-$, $u_\star$ and $u_+$ are three consecutive zeros of $g$ with $g'(u_-)<0$, $g'(u_+)<0$ and $g'(u_\star)>0$ and $\sigma=f'(u_\star)$, $f''(u_\star)\neq0$, one indeed derives spectrally stable waves in this way, that are fronts connecting $u_-$ and $u_+$ through $u_\star$. As a consequence of the foregoing discussion, note that the presence of a nondegenerate source term selects a discrete set of constant solutions, but also a discrete set of wavespeeds for stable fronts. Beyond (discontinuous or smooth) fronts and constant solutions, the equation may also support spatially periodic traveling waves. These are however necessarily discontinuous and, as a consequence of Lax's admissibility condition, each of their smooth part must also contain a sonic point (see~\cite{JNRYZ} for details, on a closely related system case).

Under rather natural assumptions on $f$ and $g$ --- including the strict convexity of $f$ and the dissipativity at infinity of $g$ ---, it has been proved that starting from an $L^\infty$ initial data that is either spatially periodic or is constant near $-\infty$ and near $\infty$, the large-time dynamics is indeed well captured in $L^\infty$ topology by piecing together traveling waves (constants, fronts or periodic waves). In the periodic setting~\cite{FanHale93,Lyberopoulos94,Sinestrari95,Sinestrari97}, every solution approaches asymptotically either a periodic (necessarily discontinuous) traveling wave, or a constant equilibrium. Moreover, periodic traveling waves are actually unstable and the rate of convergence is exponential in the latter case whereas it may be arbitrarily slow in the former case, even when restricting to solutions that are initially close and do converge to a periodic traveling wave. Starting from data with essentially compact support~\cite{Sinestrari96,MasciaSinestrari97}, the large-time asymptotics may a priori involve several blocks of different kinds (constants, fronts or periodics). Yet the scenario generating periodic blocks is also non generic and unstable. Note that at the level of regularity considered there the strict convexity assumption on $f$ plays a key role as it impacts the structure of possible discontinuities. The few contributions relaxing the convexity assumption add severe restrictions on $g$ or on the initial data, for instance linearity of $g$ in~\cite{Lyberopoulos92}, Riemann initial data in~\cite{Sinestrari97a,Mascia00} and monotonicity of the initial data in~\cite{Mascia98}.

At a technical level, one key ingredient in the proofs of the aforementioned series of investigations are generalized characteristics of Dafermos~\cite{Dafermos77}. They provide a formulation of the equation that is well suited to comparison principles thus to asymptotics in $L^\infty$ topology. Our goal here is in a neighborhood of one stable traveling wave (of a specific kind) to complete the picture with a description in stronger topologies assuming more regularity but less localization on initial data. By doing so we expect to contribute to put on a par the stability theory for~\eqref{eq-u} with the one successfully derived over the years for parabolic systems (see for instance~\cite{KapitulaPromislow-stability} for the stability of constants, fronts and solitary waves, and~\cite{JNRZ-conservation} for periodic waves). In particular, we derive our asymptotics under spectral stability assumptions that are sharp up to the exclusion of limit cases. Among the difficulties to overcome in carrying out such a general program are the absence of  regularization effects sufficiently strong to rely on a Duhamel formula based on a linearization about the reference wave and the presence of discontinuities and/or of sonic points in the profiles themselves that alter even the nature of the underlying spectral problems. 

Whereas in a companion paper~\cite{DR2} we do study waves exhibiting sonic points, we restrict here as announced 
to the stability of Riemann shocks, that is, to waves given by $\uu(t,x)=\uU(x-(\psi_0+\sigma t))$ with initial shock position $\psi_0\in\RR$, speed $\sigma\in\RR$ and wave profile $\uU$ such that
\[ \uU(x)=\begin{cases}
\uu_- &\text{ if }  x<0\\
\uu_+ &\text{ if }  x>0
\end{cases}\]
where $(\uu_-,\uu_+)\in\RR^2$, $\uu_+\neq\uu_-$.  The function $\uu$ is indeed an entropy-admissible solution provided that
\[
g(\uu_+)=0\,,\qquad g(\uu_-)=0\,,\qquad
f(\uu_+)-f(\uu_-)=\sigma (\uu_+-\uu_-)\,,
\]
and Oleinik's condition holds
\begin{equation}\label{gLax-shock-intro}
\begin{cases}
\qquad\qquad\qquad\sigma\,\geq\,f'(\uu_+)\,,&\\[0.5em]
\frac{f(\tau\,\uu_-+(1-\tau)\,\uu_+)-f(\uu_-)}{\tau\,\uu_-+(1-\tau)\,\uu_+-\uu_-}\geq
\frac{f(\tau\,\uu_-+(1-\tau)\,\uu_+)-f(\uu_+)}{\tau\,\uu_-+(1-\tau)\,\uu_+-\uu_+}&\qquad\textrm{for any }\ \tau\in(0,1)\,,\\[0.5em]
\qquad\qquad\qquad f'(\uu_-)\,\geq\,\sigma\,.
\end{cases}
\end{equation}
One may readily check that 
\begin{equation}\label{stab-spec-shock-intro}
g'(\uu_+)\leq0 \quad \text{ and } \quad g'(\uu_-)\leq0
\end{equation}
are necessary to exclude spectral instability of $\uu$. We prove asymptotic orbital stability in $W^{1,\infty}$ topology, with sharp exponential decay rates and asymptotic phase, under $BUC^1$ perturbations possibly jointly with perturbations on the position and the strength of the discontinuity jump when~\eqref{gLax-shock-intro} and~\eqref{stab-spec-shock-intro} hold with strict inequalities. Likewise, we also provide stability results under $BUC^k$ perturbations for any $k\geq1$. We stress that at this stage no convexity assumption is needed. Yet, our approach may also be extended to cases when perturbations are only piecewise $BUC^1$ with a finite number\footnote{Yet for exposition purposes, we only provide details about the case where this number is at most one.} of discontinuities of shock-type, and then we do assume that $f''(\uu_-)\neq0$ and $f''(\uu_+)\neq0$ (or only half of it if shock-type discontinuities are only introduced on one side of the reference discontinuity).

One important point contrasting with the purely conservative case is that near $\uu$ the positions of discontinuities arising from piecewise smooth perturbations with smooth parts sufficiently small in $BUC^1$ may be predicted at leading order from the linearized dynamics. This may be intuited by analogy from the consideration of solutions near $\uu\equiv0$ to 
\[
\d_t u+\d_x\left(\alpha\frac{u^2}{2}\right)=-\beta\,u
\]
with $\alpha\in\RR$, $\beta\geq0$. On the latter basic explicit example, by studying $\d_xu$ along characteristics, one readily checks that the existence of a classical solution and the persistence of regularity holds globally forward-in-time if and only if $\alpha\,\d_x(u(0,\cdot))\geq -\beta$. Hence when $\beta>0$ and $\alpha\neq0$, shock formation may be prevented by assuming asymmetric initial smallness on the derivative of the initial data. Incidentally note that this asymmetry is fundamental in~\cite{Mascia98}. This also hints at a classification of discontinuities in initial data between shock-like discontinuities across which $f'$ decreases and rarefaction-like discontinuities across which $f'$ increases. The latter are removable by a density argument in the sense that the generated dynamics may be approximated by the one arising from a family of initial data where the discontinuity is absent. In particular, provided results are proved under sharp asymmetric smallness conditions, there is no loss in generality in assuming that any discontinuity is of shock-type.

Though we hope that similar analyses could be carried out in some system\footnote{During the finalization of the present contribution we have been informed that a system case has been analyzed in~\cite{YZ} with distinct but not disjoint techniques. Parts of the arguments used in~\cite{YZ} actually originate in private communications of the second author of the present contribution to the second author of~\cite{YZ}.} cases, we use here crucially the scalar structure to analyze the evolution of the piecewise regularity in the following way. First we extend each smooth part of the initial data to a function on $\RR$, that is either close to $\uu_-$ or close to $\uu_+$. Then we propagate each of the extended initial data and achieve suitable estimates on the corresponding dynamics near stable constant states. At last we use the evolved extensions to determine the evolution of shock locations by solving the corresponding Rankine-Hugoniot conditions and glue them along the shock curves to obtain the solution for the original discontinuous initial data. In particular along the way in order to carry out the second step we prove a $BUC^1$ asymptotic stability result for constant solutions $\uu$, that is, constant functions with value a zero $\uu$ of $g$, such that $g'(\uu)<0$. Though in principle the foregoing result could be proved ---yet much less readily than $L^\infty$ asymptotics--- with classical characteristics and comparison principles\footnote{Similar results could also be obtained by energy estimates provided one relaxes the essentially sharp $BUC^1$ framework to the $L^2$-based $H^2$ space.} (along the lines in~\cite[Chapter~4]{Li94}), we choose to use tools as close as possible to those in the classical stability theory~\cite{KapitulaPromislow-stability,JNRZ-conservation}, relying on resolvent estimates and semigroup theory. However, as mentioned hereinabove, since regularization effects are too weak, it is not sufficient to consider the linearized dynamics. Instead we prove that spectral assumptions yield decay estimates for all nearby ---time and space dependent--- linear dynamics, hence actually use the evolution system (see~\cite[Chapter~5]{Pazy}) rather than semigroup framework.

In the rest of the present paper, we first study the asymptotic stability of constant states under regular perturbations in $BUC^1(\RR)$, as stated in Section~\ref{s:constant-shockless} and proved in Sections~\ref{s:linear} and~\ref{s:proof}. In Section~\ref{s:constant_by_shock} we extend our analysis to the case where a constant state is perturbed by a (small) shock. Then we turn to our main concern, the asymptotic stability of (large) Riemann shocks, under perturbations that are either regular (Section~\ref{s:shockless}) or piecewise regular with a small shock (Section~\ref{s:large_and_small_shocks}). Next, in Section~\ref{s:multiD} we investigate extensions to multidimensional settings. At last, in Section~\ref{s:conclusion}, we provide some further insights on limitations and extensions of the present analysis. 

\section{Asymptotic stability of constant states}\label{S:constant}

\subsection{Asymptotic stability under shockless perturbations}\label{s:constant-shockless}

In this section, first we show the asymptotic stability of constant states with respect to regular perturbations under the natural spectral condition. 

\begin{Proposition}\label{P.constant-conditional}
Let $g\in \cC^2(\RR)$ and $\uu\in\RR$ be such that 
\begin{equation}\label{stab-spec-constant}
g(\uu)=0\qquad\textrm{and}\qquad g'(\uu)<0\,.
\end{equation}
Then for any $C_0>1$, there exists $\epsilon>0$ such that for any $f\in \cC^2(\RR)$, for any $v_0\in BUC^1(\RR)$ satisfying 
\[\Norm{v_0}_{L^{\infty}(\RR)}\leq \epsilon\,,\]
the unique maximal classical solution to~\eqref{eq-u}, $u\in\cC^0([0,T_*(v_0));BUC^1(\RR))\cap \cC^1([0,T_*(v_0));BUC^0(\RR))$ with $T_*(v_0)\in(0,\infty]$, generated by the the initial data $u\id{t=0}=\uu+v_0$ satisfies for any ${0\leq t<T_*(v_0)}$
\[
\Norm{u-\uu}_{L^{\infty}(\RR)} \leq \Norm{v_0}_{L^{\infty}(\RR)}
C_0\,e^{g'(\uu)\,t}
\]
and if moreover $\d_xv_0\in L^1(\RR)$
\[\Norm{\d_xu(t,\cdot)}_{L^1(\RR)} \leq \Norm{\d_xv_0}_{L^1(\RR)} 
C_0\,e^{g'(\uu)\,t}\,.\]
\end{Proposition}

The foregoing proposition is a \emph{conditional} asymptotic stability result. Proximity is guaranteed only as long as the solution persists as a classical solution. A strong sign that the result tells nothing about persistence of regularity is that the required smallness is independent of $f$ and does not involve derivatives of $v_0$. This should be contrasted with the explicit example discussed in the introduction. 

In a framework involving a smallness condition with more regularity, one may prove

\begin{Proposition}\label{P.constant-classical}
Let $f,\,g\in \cC^2(\RR)$ and $\uu\in\RR$ be such that 
\[
g(\uu)=0\qquad\textrm{and}\qquad g'(\uu)<0\,.
\]
Then for any $C_0>1$, there exists $\epsilon>0$ such that for any $v_0\in BUC^1(\RR)$ satisfying 
\[\Norm{v_0}_{W^{1,\infty}(\RR)}\leq \epsilon\,,\]
the initial data $u\id{t=0}=\uu+v_0$ generates a global unique classical solution to~\eqref{eq-u}, ${u\in BUC^1(\RR^+\times\RR)}$, and it satisfies for any $t\geq0$
\begin{align*}
\Norm{u(t,\cdot)-\uu}_{L^{\infty}(\RR)}&\leq \Norm{v_0}_{L^{\infty}(\RR)}C_0\,e^{g'(\uu)\,t}\ ;\\
\Norm{\d_x u(t,\cdot)}_{L^{\infty}(\RR)}&\leq \Norm{\d_x v_0}_{L^{\infty}(\RR)} C_0\,e^{g'(\uu)\,t}\,.
\end{align*}
\end{Proposition}

Assuming local convexity/concavity, one may relax part of the foregoing smallness condition

\begin{Proposition}\label{P.constant}
Let $f,\,g\in \cC^2(\RR)$ and $\uu\in\RR$ be such that
\[g(\uu)=0\,,\qquad g'(\uu)<0\qquad\textrm{and}\qquad f''(\uu)\neq0\,.\]
Then for any $C_0>1$, there exists $\epsilon>0$ such that for any $v_0\in BUC^1(\RR)$ satisfying 
\[\Norm{v_0}_{L^{\infty}(\RR)}\leq \epsilon
\qquad\textrm{and}\qquad \Norm{(\sign(f''(\uu))\,\d_xv_0)_-}_{L^{\infty}(\RR)}\,\leq\epsilon\,,
\]
the initial data $u\id{t=0}=\uu+v_0$ generates a global unique classical solution to~\eqref{eq-u}, ${u\in BUC^1(\RR^+\times\RR)}$, and it satisfies for any $t\geq0$
\begin{align*}
\Norm{u(t,\cdot)-\uu}_{L^{\infty}(\RR)}&\leq \Norm{v_0}_{L^{\infty}(\RR)} C_0\,e^{g'(\uu)\,t}\,,\\
\Norm{(\sign(f''(\uu))\,\d_x u(t,\cdot))_-}_{L^{\infty}(\RR)}&\leq \Norm{(\sign(f''(\uu))\,\d_xv_0)_-}_{L^{\infty}(\RR)} C_0\,e^{g'(\uu)\,t}\,,\\
\Norm{\d_x u(t,\cdot)}_{L^{\infty}(\RR)}&\leq \Norm{\d_x v_0}_{L^{\infty}(\RR)} C_0\,e^{g'(\uu)\,t}\,.
\end{align*}
\end{Proposition}

By a classical approximation/compactness argument one then deduces

\begin{Corollary}\label{C.constant}
Let $f,\,g\in \cC^2(\RR)$ and $\uu\in\RR$ be such that
\[g(\uu)=0\,,\qquad g'(\uu)<0\qquad\textrm{and}\qquad f''(\uu)\neq0\,.\]
Then  for any $C_0>1$, there exists $\epsilon>0$ such that for any $v_0\in BV_{loc}(\RR)\cap L^\infty(\RR)$ such that $(\sign(f''(\uu))\,\d_xv_0)_-\in L^\infty(\RR)$ and
\[\Norm{v_0}_{L^{\infty}(\RR)}\leq \epsilon
\qquad\textrm{and}\qquad \Norm{(\sign(f''(\uu))\,\d_xv_0)_-}_{L^{\infty}(\RR)}\,\leq\epsilon\,,
\]
the initial data $u\id{t=0}=\uu+v_0$ generates a global unique entropy solution to~\eqref{eq-u} and it satisfies for a.e. $t\geq0$
\begin{align*}
\Norm{u(t,\cdot)-\uu}_{L^{\infty}(\RR)}&\leq \Norm{v_0}_{L^{\infty}(\RR)} C_0\,e^{g'(\uu)\,t}\,,\\
\Norm{(\sign(f''(\uu))\,\d_x u(t,\cdot))_-}_{L^{\infty}(\RR)}&\leq \Norm{(\sign(f''(\uu))\,\d_xv_0)_-}_{L^{\infty}(\RR)} C_0\,e^{g'(\uu)\,t}\,,
\end{align*}
and if moreover $v_0\in BV(\RR)$
\[
\Norm{u(t,\cdot)}_{TV(\RR)}\,\leq\,\Norm{v_0}_{TV(\RR)} C_0\,e^{g'(\uu)\,t}\,,
\]
while if $\d_xv_0\in L^\infty(\RR)$
\[
\Norm{\d_xu(t,\cdot)}_{L^\infty(\RR)}\,\leq\,\Norm{\d_xv_0}_{L^\infty(\RR)} C_0\,e^{g'(\uu)\,t}\,.
\]
\end{Corollary}

To enlighten the content of Corollary~\ref{C.constant}, we stress that it allows discontinuous initial data generating small rarefaction waves but not shocks. This does not mean that a similar result cannot hold when small shocks are present but simply that in general, as the explicit example of the introduction shows, they cannot be obtained by a limiting process building on global classical solutions. This is consistent with expectations drawn from general theory, see for instance~\cite[Chapter~9, Problem~6]{Bressan}.

\begin{Remark}
An examination of proofs shows that one may relax everywhere the assumption that $g\in\cC^2$. It is sufficient that $g\in\cC^1$ and that the modulus of continuity 
\[
\omega(r)\,=\,\max_{|u-\uu|\leq r}\|g'(u)-g'(\uu)\|
\]
is such that $r\mapsto \omega(r)/r$ is locally integrable. This includes the case when $g\in\cC^{1,\alpha}$, $\alpha>0$. Indeed the key property is that for any positive $C$ and $\theta$
\[
\int_0^\infty \omega(C\,\varepsilon\,e^{-\theta\,t})\,\dd t
\,=\,\frac{1}{\theta}\int_{0}^{C\,\varepsilon}\,\frac{\omega(r)}{r}\,\dd r\ 
\stackrel{\varepsilon\to0^+}{\longrightarrow}\ 0\,.
\]
\end{Remark}

The exponential decay in time also holds for higher order derivatives without further restriction on sizes of perturbations.
\begin{Proposition}\label{P.high-order} Under the assumptions of either Proposition~\ref{P.constant-classical} or Proposition~\ref{P.constant}, if one assumes additionally that $f\in\cC^{k+1}(\RR)$, $g\in \cC^k(\RR)$ with $k\in\NN$, $k\geq 2$ then there exists $C_k>0$, depending on $f$, $g$ and $k$ but not on the initial data $v_0$, such that if $v_0\in BUC^k(\RR)$ additionally to constraints in either Proposition~\ref{P.constant-classical} or Proposition~\ref{P.constant}, then the global unique classical solution to~\eqref{eq-u} emerging from the initial data $\uu+v_0$ satisfies $u\in BUC^k(\RR^+\times\RR)$ and for any $t\geq0$
\[
\Norm{\d_x^k u(t,\cdot)}_{L^{\infty}(\RR)}\leq 
\Norm{\d_x^k v_0}_{L^{\infty}(\RR)}e^{C_k\,\|v_0\|_{W^{1,\infty}}\,(1+\|v_0\|_{W^{1,\infty}}^{k-1})}\,e^{g'(\uu)\,t}\,.
\]
\end{Proposition}

The local well-posedness theory for~\eqref{eq-u} at the various levels of regularity considered here is standard. Note in particular that in the foregoing statements without any further constraint uniqueness holds also on any finite time interval. Though we shall not repeat it henceforth this remark applies equally well to all our uniqueness claims. Thus the main upshots of Propositions~\ref{P.constant-classical} and~\ref{P.constant}, Corollary~\ref{C.constant} and Proposition~\ref{P.high-order} are global existence of classical solutions and exponential decay in time. For the classical well-posedness theory for scalar balance laws, due to Kru\v zkov, the reader is referred to~\cite{Kruzhkov} and\footnote{Unfortunately, as most of textbooks, for expository reasons~\cite[Chapter~6]{Bressan} restricts to conservation laws. Yet for local-in-time issues, such as well-posedness, changes needed to extend from conservation laws to balance laws are immaterial.}~\cite[Chapter~6]{Bressan}.

For our purposes it is expedient to introduce $v\eqdef u-\uu$ and as long as classical solutions are concerned work with the following quasilinear form of~\eqref{eq-u}
\begin{equation}\label{eq.linearized-constant}
\d_t v+f'(\uu+v)\d_x v-g'(\uu) v=g(\uu+ v)-g(\uu)-g'(\uu) v.
\end{equation}
Note in particular that in the above formulation one cannot allow any ``regularity loss" due to a linearization of the transport term. Bearing this in mind, prior to the consideration of a mild formulation of~\eqref{eq.linearized-constant} we analyze linear equations of the form 
\begin{equation}\label{eq.frozen-constant}
\d_t v+a\d_x v-bv=r
\end{equation}
where $a$ is close to $f'(\uu)$ and $b$ close to $g'(\uu)$ in a suitable sense. Let us anticipate that to deal with the mild formulation of~\eqref{eq.linearized-constant} and prove Propositions~\ref{P.constant-conditional} and~\ref{P.constant-classical} we could stick to the case where $b=g'(\uu)$. We shall use the extra flexibility in the choice of $b$ only when tracking asymmetric regularity involved in Proposition~\ref{P.constant} and Corollary~\ref{C.constant}.

As a preliminary let us discuss the linearized equation
\[
\d_t v+f'(\uu)\d_x v-g'(\uu) v=0\,.
\]
A notion of solution may be obtained through the classical semigroup formalism. For instance one may consider $L=-f'(\uu)\d_x+g'(\uu)$ on either $L^p(\RR)$, $1\leq p<\infty$, or $BUC^0(\RR)$ with domain $W^{1,p}(\RR)$ or $BUC^1(\RR)$, if $f'(\uu)\neq0$ and $L^p(\RR)$ or $BUC^0(\RR)$, otherwise.  The operator $L$ is then closed densely-defined with spectrum $g'(\uu)+i\,\RR$ if $f'(\uu)\neq0$, $\{g'(\uu)\}$ otherwise. In particular, $g'(\uu)>0$ would yield spectral instability whereas as follows from the analysis below suitable resolvent estimates show that $g'(\uu)<0$ provides linear asymptotic stability with exponential rates. We refer the reader to~\cite{Pazy,vanNeerven} for background on semigroups and their large-time behaviors.

It is already apparent here that though this does not alter significantly the stability properties, the vanishing of transport term impacts dramatically the regularity structure of the spectral problem. As long as we restrict to classical solutions near a constant steady state this is immaterial since going to a uniformly moving frame may remove possible vanishings. This would however not be possible near the continuous stable traveling fronts described in the introduction. In general the presence of an essential characteristic point is a serious cause of trouble, and the reader is referred to~\cite{JNRYZ,DR2} for an example of its impact on spectral problems.

As a consequence it is convenient to change coordinate frame. Explicitly for any $\sigma\in\RR$, by introducing $\tv$ through $\tv(t,x)=v(t,x+\sigma t)$ one replaces~\eqref{eq.frozen-constant} with 
\[
\d_t \tv+(\ta-\sigma)\d_x \tv-\tb\tv=\tr\,,
\]
with $(\ta,\tb,\tr)$ defined by $(\ta,\tb,\tr)(t,x)=(a,b,r)(t,x+\sigma t)$. Implicitly some of our assumptions on $a$ will build on the fact that one may choose $\sigma$ so that $\ta-\sigma$ is bounded away from zero.

\subsection{Linear equations}\label{s:linear}

To consider~\eqref{eq.frozen-constant} with time-dependent $a$ and $b$, we may either rely on or mimic the available abstract theory for evolution systems, as described in~\cite[Chapter~5]{Pazy}. In any case the needed elementary block is the solution of problems where $a$ and $b$ are independent of time. 

As a consequence we first consider this case. With this restriction we are back to the semigroup framework that may be analyzed directly by resolvent estimates. In the present section we always assume that $a\,,\ b\in BUC^0(\RR)$ with $a$ bounded away from zero. For such $a$, $b$, $L_{a,\,b}=-a\d_x+b$ is elliptic\footnote{Or, in a more standard terminology, $i\,L_{a,\,b}$ is elliptic.}, and is a closed, densely-defined operator on either $L^p(\RR)$ with domain $W^{1,p}(\RR)$, $1\leq p<\infty$, or on $BUC^0(\RR)$ with domain $BUC^1(\RR)$. The key basic estimate is 

\begin{Lemma}\label{l:resolvent}
Assume $a,\,b\in BUC^0(\RR)$ with $a$ bounded away from zero.\\
{\bf (i).} Then for any $\lambda\in\CC$ such that 
\[
\Re(\lambda)>\sup_\RR b(\cdot)\,,\] 
for any $F\in BUC^0(\RR)$, there exists a unique $\check{v}(\,\cdot\,;\lambda)\in BUC^1(\RR)$ such that 
\[
(\lambda-L_{a,\,b})\,\check{v}(\,\cdot\,;\lambda)\,=\,F
\]
and moreover
\[
\Norm{\check v(\,\cdot\,;\lambda)}_{L^\infty(\RR)}\leq \frac1{\Re\lambda-\sup_\RR b(\cdot)}\Norm{F}_{L^\infty(\RR)}\,.
\]
If $b$ is constant and $F\in W^{1,1}(\RR)$, then $\check{v}(\,\cdot\,;\lambda)\in W^{1,1}(\RR)$ and
\[
\Norm{\d_x\check v(\,\cdot\,;\lambda)}_{L^1(\RR)}\leq \frac1{\Re\lambda-b}\Norm{\d_xF}_{L^1(\RR)}\,.
\]
Moreover if $\lambda\in\RR$, $\lambda\in(\sup_\RR b(\cdot),\infty)$ and $F\geq0$ then $\check v(\,\cdot\,;\lambda)\geq0$.\\
{\bf (ii).} Assume moreover that 
\[
a\in BUC^1(\RR)\,,\qquad
b\textrm{ is constant}\,,\qquad
\Re(\lambda)>b-\inf_\RR a'(\cdot)\,,
\]
and $F\in W^{1,\infty}(\RR)$. Then
\[
\Norm{\d_x\check v(\,\cdot\,;\lambda)}_{L^\infty(\RR)}\leq \frac1{\Re\lambda-b+\inf_\RR a'(\cdot)}\Norm{\d_xF}_{L^\infty(\RR)}\,.
\]
\end{Lemma}

\begin{proof}
Let us begin with the uniqueness part. If $(\lambda-L_{a,\,b})\,\check{v}(\,\cdot\,;\lambda)=0$ then actually 
\[
\check{v}(\,x\,;\lambda)\,=\,e^{ \int_0^x\frac{b(z)-\lambda}{a(z)}\,\dd z}\,\check{v}_0
\]
for some constant $\check{v}_0\in\CC$. Then if $a$ is positive and bounded away from zero and $\Re(\lambda)>\sup_\RR b(\cdot)$, the boundedness near $x=-\infty$ implies $\check{v}_0=0$ since $\abs{e^{ \int_0^x\frac{b(z)-\lambda}{a(z)}\,\dd z}}\geq e^{|x|\,\frac{\Re(\lambda)-\sup_\RR b(\cdot)}{\|a\|_{L^\infty}}}$ when $x<0$. Likewise if $a$ is negative and bounded away from zero and $\Re(\lambda)>\sup_\RR b(\cdot)$, boundedness near $x=\infty$ yields $\check{v}_0=0$.

From now on for definiteness we assume that $a$ is positive and bounded away from zero. Note that there is no loss of generality since one may go from this case to the opposite one by reversing $x$ into $-x$.

Let $F\in BUC^0(\RR)$. One readily checks when $\Re(\lambda)>\sup_\RR b(\cdot)$ that
\[
\check v(x;\lambda)\eqdef \int_{-\infty}^x e^{ \int_y^x\frac{b(z)-\lambda}{a(z)}\,\dd z}\frac{F(y)}{a(y)}\,\dd y
\]
defines $\check{v}(\,\cdot\,;\lambda)\in BUC^1(\RR)$ and that
\begin{align*}
\abs{\check v(x;\lambda)}&\leq \frac{\Norm{F}_{L^\infty(\RR)}}{\Re\lambda-\sup_\RR b(\cdot)} 
\int_{-\infty}^x e^{ \int_y^x\frac{b(z)-\Re\lambda}{a(z)}\,\dd z}\frac{\Re\lambda-b(y)}{a(y)} \,\dd y=\frac{\Norm{F}_{L^\infty(\RR)}}{\Re\lambda-\sup_\RR b(\cdot)}\,,\\
\d_x\check v(x;\lambda)&=\frac{F(x)}{a(x)}+  \int_{-\infty}^{x} \frac{b(x)-\lambda}{a(x)} e^{ \int_y^x\frac{b(z)-\lambda}{a(z)}\,\dd z}  \frac{F(y)}{a(y)}\,\dd y\,.
\end{align*}
It is also straightforward to check that if moreover $\lambda\in\RR$ and $F\geq0$ then $\check v(\cdot;\lambda)\geq0$. When moreover $b$ is constant and $\d_xF\in L^1(\RR)+L^\infty(\RR)$, the latter expression may be integrated by parts into
\begin{align*}
\d_x\check v(x;\lambda)
&=  \int_{-\infty}^{x} e^{ \int_y^x\frac{b-\lambda}{a(z)}\,\dd z}  \frac{\d_y F(y)}{a(x)}\,\dd y\\
&=  \int_{-\infty}^{x} e^{ \int_y^x\frac{b-\lambda-a'(z)}{a(z)}\,\dd z}  \frac{\d_y F(y)}{a(y)}\,\dd y\,.
\end{align*}
The latter expression may be used to obtain the $\dot{W}^{1,\infty}\to\dot{W}^{1,\infty}$ bound as we have derived the $L^\infty\to L^\infty$ bound. Concerning the former expression it may be integrated in $x$ to deduce the $\dot{W}^{1,1}\to\dot{W}^{1,1}$ bound since when $b$ is constant for any $y$
\[
\int_y^{\infty} e^{ \int_y^x\frac{b-\Re(\lambda)}{a(z)}\,\dd z}  \frac{\dd x}{a(x)}
\,=\,\frac{1}{\Re(\lambda)-b}\,.
\]
\end{proof}

With the above frozen-time resolvent estimates, for general coefficients $a,b$ one may first change frame to ensure that $a$ is bounded away from zero then apply general theorems on evolution systems. See for instance~\cite[Chapter~5, Theorem~3.1]{Pazy} with $X=BUC^0(\RR)$ and $Y=BUC^1(\RR)$, and apply~\cite[Chapter~5, Theorem~2.3]{Pazy} to reduce the verification of assumption~$(H_2)$ there to the case where $b$ is constant.

\begin{Proposition}\label{P.estimate-semilinear-constant}
Let $T\in(0,\infty]$, $a\in\cC^0([0,T),BUC^1(\RR))$, $b \in\cC^0([0,T),BUC^0(\RR))$. Then the family of operators $\cL_t=L_{a(t,\cdot),\,b(t,\cdot)}$ generates an evolution system $\cS_{a,b}$ on $BUC^0(\RR)$ such that for any $v_0\in BUC^0(\RR)$, any $0\leq s\leq t<T$
\[\Norm{\cS_{a,b}(s,t)\,v_0}_{L^\infty(\RR)}
\leq e^{\int_s^t\sup_\RR b(\tau,\cdot)\,\dd\tau}\Norm{v_0}_{L^\infty(\RR)}\,,\]
and $v_0\geq0$ implies $\cS_{a,b}(s,t)\,v_0\geq0$. If moreover $b$ is constant, then $v_0 \in BUC^0(\RR)\cap W^{1,1}(\RR)$ yields for any $0\leq s\leq t<T$
\[
\Norm{\d_x \cS_{a,b}(s,t)\,v_0}_{L^1(\RR)}
\leq e^{(t-s)\,b}\Norm{\d_x v_0}_{L^1(\RR)}\,,\]
and from $v_0 \in  W^{1,\infty}(\RR)$ stems for any $0\leq s\leq t<T$,
\[\Norm{\d_x \cS_{a,b}(s,t)\,v_0}_{L^\infty(\RR)}
\leq e^{(t-s)\,b-\int_s^t\inf_\RR \d_xa(\tau,\cdot)\,\dd\tau}\Norm{\d_xv_0}_{L^\infty(\RR)}\,.
\]
\end{Proposition}

\subsection{The shockless nonlinear problem}\label{s:proof}

In this section we complete the proofs of results from Section~\ref{s:constant-shockless}. 

\begin{proof}[Proof of Proposition~\ref{P.constant-conditional}.] 
Let $\epsilon\in(0,1]$. Pick a classical solution $u=\uu+v$ starting from $\uu+v_0$ such that $\Norm{v_0}_{L^\infty(\RR)}\leq\epsilon$. Then if $u$ exists (as a classical solution) on $[0,t_0)$, for any $0\leq t<t_0$
\[
v(t,\cdot)\,=\,\cS_{f'(\uu+v),\,g'(\uu)}v_0
+\int_0^t \cS_{f'(\uu+v),\,g'(\uu)}(s,t)\big(g(\uu+v)-g(\uu)-g'(\uu)v\big)(s,\cdot) \,\dd s\,.
\]
Therefore if moreover for any $t\in [0,t_0)$, $\Norm{v(t,\cdot)}_{L^\infty(\RR)}\leq 2\epsilon\,e^{g'(\uu)\,t}$, then for any $t\in [0,t_0)$
\[
e^{-g'(\uu)\,t}\Norm{v(t,\cdot)}_{L^\infty(\RR)}
\leq \Norm{v_0}_{L^\infty(\RR)}
+2\epsilon\,C_g\,\int_0^t e^{g'(\uu)\,s}\,\left(e^{-g'(\uu)\,s}\Norm{v(s,\cdot)}_{L^\infty(\RR)}\right)\dd s
\]
where $C_g=\tfrac12\Norm{g''}_{L^\infty([\uu-2\epsilon,\uu+2\epsilon])}$, so that for any $t\in [0,t_0)$,
\begin{equation}\label{eq:Linfty}
\Norm{v(t,\cdot)}_{L^\infty(\RR)}
\leq \Norm{v_0}_{L^\infty(\RR)}\,e^{g'(\uu)\,t}\,e^{2\epsilon\,C_g\int_0^te^{g'(\uu)\,s}\dd s}
\leq \Norm{v_0}_{L^\infty(\RR)}\,e^{g'(\uu)\,t}\,e^{\epsilon\,\frac{2C_g}{|g'(\uu)|}}\,.
\end{equation}
If $\epsilon$ is small enough to ensure that $\exp(\epsilon\,\frac{2C_g}{|g'(\uu)|})<2$ then a continuity argument yields that estimate~\eqref{eq:Linfty} holds as long as $u$ persists as a classical solution. Since $\exp(\epsilon\,\frac{2C_g}{|g'(\uu)|})$ may be brought arbitrarily close to $1$ by choosing $\epsilon$ small, this proves the $L^\infty$ part of Proposition~\ref{P.constant-conditional}. With this bound in hands we deduce even more directly that if moreover $\d_xv_0\in L^1$ then
\[
\Norm{\d_xv(t,\cdot)}_{L^1(\RR)}
\leq \Norm{\d_xv_0}_{L^1(\RR)}\,e^{g'(\uu)\,t}\,e^{\epsilon\,\frac{2C_0\,C_g}{|g'(\uu)|}}\,.
\]
This achieves the proof by taking $\epsilon$ even smaller if needed.
\end{proof}

The proof of Proposition~\ref{P.constant-classical} being completely similar, we omit it.

\begin{proof}[Proof of Proposition~\ref{P.constant}.]
First we fix $\epsilon\in(0,1]$ sufficient small to satisfy conclusions of Proposition~\ref{P.constant-conditional} and to ensure that on $[\uu-C_0\epsilon,\uu+C_0\epsilon]$, $f''$ is of the sign of $f''(\uu)$. To proceed we use that if $u=\uu+v$ persists as a classical solution on $[0,t_0)$ then for $t\in[0,t_0)$
\[
\d_xv(t,\cdot)=\cS_{f'(\uu+v),\,g'(\uu+v)-f''(\uu+v)\d_xv}(0,t)\,\d_xv_0
\] 
thus by linearity and preservation of non negativity
\[
(\sign(f''(\uu))\d_xv(t,\cdot))_-\leq\cS_{f'(\uu+v),\,g'(\uu+v)-f''(\uu+v)\d_xv}(0,t)\,(\sign(f''(\uu))\d_xv_0)_-\,.
\] 
Therefore if moreover for any $t\in [0,t_0)$, $\Norm{(\sign(f''(\uu))\d_xv(t,\cdot))_-}_{L^\infty(\RR)}\leq 2\epsilon\,e^{g'(\uu)\,t}$, then for any $t\in [0,t_0)$
\[
\Norm{(\sign(f''(\uu))\d_xv(t,\cdot))_-}_{L^\infty(\RR)}
\leq \Norm{(\sign(f''(\uu))\d_xv_0)_-}_{L^\infty(\RR)}
e^{g'(\uu)\,t}
e^{\epsilon\frac{2C_f+C_0\,C_g}{|g'(\uu)|}}
\]
with $C_f=\Norm{f''}_{L^\infty([\uu- C_0\epsilon,\uu+ C_0\epsilon])}$ and $C_g=\Norm{g''}_{L^\infty([\uu- C_0\epsilon,\uu+ C_0\epsilon])}$. By choosing $\epsilon$ sufficiently small so that $e^{\epsilon\frac{2C_f+C_0\,C_g}{|g'(\uu)|}}\leq \min(\{2,C_0\})$, one deduces that if $u=\uu+v$ persists as a classical solution on $[0,t_0)$ then for $t\in[0,t_0)$
\[
\Norm{(\sign(f''(\uu))\d_xv(t,\cdot))_-}_{L^\infty(\RR)}
\leq \Norm{(\sign(f''(\uu))\d_xv_0)_-}_{L^\infty(\RR)}
C_0\,e^{g'(\uu)\,t}\,.
\]
One concludes by noticing that this implies that 
if $u=\uu+v$ persists as a classical solution on $[0,t_0)$ then for $t\in[0,t_0)$
\[
\Norm{\d_xv(t,\cdot)}_{L^\infty(\RR)}
\leq \Norm{\d_xv_0}_{L^\infty(\RR)}
C_0\,e^{g'(\uu)\,t}\,,
\]
which rules out finite-time blow-up.
\end{proof}

\begin{proof}[Proof of Proposition~\ref{P.high-order}.]
Propagation of regularity being classical, we focus on the decay estimate. Note that since we already know that $v$ is small in $L^\infty$
\begin{align*}
\d_x^k v(t,\cdot)&=\cS_{f'(\uu+v),\,g'(\uu)}(0,t)\,\d_x^kv_0\\
&+\int_0^t \cS_{f'(\uu+v),\,g'(\uu)}(s,t)
\Big(c_0(v(s,\cdot))\,v(s,\cdot)\,\d_x^kv(s,\cdot)+\hspace{-2.5em}\sum_{\substack{2\leq m\leq |\alpha|\\\alpha\in(\NN^*)^{m},\ |\alpha|\in\{k,k+1\}}}\hspace{-2em}c_\alpha(v(s,\cdot))\,\prod_{i=1}^m\d_x^{\alpha_i}v(s,\cdot)\Big) \dd s
\end{align*}
with $c_0$, $c_\alpha$ bounded. Note moreover that for any $1\leq\ell\leq k$, for some $C\geq0$ and any function $w$
\[
\Norm{\d_x^\ell w}_{L^\infty(\RR)}
\leq C\,\min(\{\Norm{w}_{L^\infty(\RR)}^{\frac{k-\ell}{k}}\Norm{\d_x^kw}_{L^\infty(\RR)}^{\frac{\ell}{k}},\ 
\Norm{\d_xw}_{L^\infty(\RR)}^{\frac{k-\ell}{k-1}}\Norm{\d_x^kw}_{L^\infty(\RR)}^{\frac{\ell-1}{k-1}}
\})
\]
so that for any $2\leq m\leq k+1$, $\alpha\in(\NN^*)^{m}$, $|\alpha|\in\{k,k+1\}$, there exists $C'$ and $C''$ such that for any $w$
\begin{align*}
\Norm{\prod_{i=1}^m\d_x^{\alpha_i}w}_{L^\infty(\RR)}
&\leq C'\,\min(\{\Norm{w}_{L^\infty(\RR)}^{m-\frac{|\alpha|}{k}}\Norm{\d_x^kw}_{L^\infty(\RR)}^{\frac{|\alpha|}{k}},\ 
\Norm{\d_xw}_{L^\infty(\RR)}^{m-\frac{|\alpha|-m}{k-1}}\Norm{\d_x^kw}_{L^\infty(\RR)}^{\frac{|\alpha|-m}{k-1}} \})\\
&\leq C''\,\Norm{w}_{W^{1,\infty}}^{m-1}\,\Norm{\d_x^kw}_{L^\infty(\RR)}\,.
\end{align*}
The proof is then concluded by first invoking the bounds of either Proposition~\ref{P.constant-classical} or Proposition~\ref{P.constant} jointly with those of Proposition~\ref{P.estimate-semilinear-constant} then applying the Gr\"onwall lemma. 
\end{proof}

\begin{proof}[Proof of Corollary~\ref{C.constant}.]
An initial data as in Corollary~\ref{C.constant} may be approximated through cut-off with sufficiently slow cut-off functions and convolution with positive kernels by initial data satisfying constraints of Proposition~\ref{P.constant}. Bounds of Propositions~\ref{P.constant-conditional} and~\ref{P.constant}, jointly with equation~\eqref{eq-u}, are then sufficient to extract a subsequence converging pointwise and uniformly bounded. With the latter one may take limits in weak formulations encoding the notion of entropy solution, hence proving the existence of an entropy solution starting from the prescribed initial data and satisfying claimed bounds. We refer the reader to~\cite[Section~6.2]{Bressan} for details on the latter compactness arguments.
\end{proof}

\subsection{Perturbation by small shocks}\label{s:constant_by_shock}

In this section we extend Proposition~\ref{P.constant} to the case where the perturbation contains a shock. 

We provide a description of the solution $u$ as regular on
\[
\Omega^\psi\eqdef\RR_+\times\RR\setminus\{\,(t,\psi(t))\,|\,t\geq0\,\}
\]
where $\psi$ follows the position of the shock. 

\begin{Remark}
It may be convenient to think of $u$ as being of the form
\[u:(t,x)\mapsto \tu(t,x-\psi(t))\]
with smooth unknowns $\psi:\RR_+\to\RR$ and $\tu:\RR_+\times\RR^\star\to\RR$. Though we shall not use this form explicitly here (partly because it is not convenient when two shocks are present), it underlies our strategy and statements. In particular, henceforth $\d_x u$ will not denote the distributional derivative of $u\in\mathcal{D}'(\RR)$ but its smooth part
\[\d_x u:(t,x)\mapsto \d_x\tu(t,x-\psi(t))\,.\]
Similarly, for $k\in\NN^\star$ and $\Omega$ an open domain, we denote $W^{k,\infty}(\Omega)$ (resp. $BUC^k(\Omega)$), the space of functions such that the restrictions to each connected component of $\Omega$, $\tilde \Omega$, belongs to $W^{k,\infty}(\check\Omega)$ (resp. $BUC^k(\check\Omega)$). These spaces are endowed with the obvious norms consistent with this definition.
\end{Remark}

For such a $u$ to satisfy the equation in distributional sense we require $u$ to satisfy it in a classical sense on $\Omega^\psi$ and that also holds the Rankine-Hugoniot condition, for any $t\geq0$
\[
 f(u_r(t))-f(u_l(t))=\psi'(t)  (u_r(t)-u_l(t))
\]
where $u_l(t)=\lim_{\delta\searrow 0} u(t,\psi(t)-\delta)$ and $u_r(t)=\lim_{\delta\searrow 0} u(t,\psi(t)+\delta)$. Moreover when $f''(\uu)\neq0$ then if $u$ is sufficiently close to $\uu$, its admissibility as an entropy solution is equivalent to Lax's condition~\cite[Section~4.5]{Bressan}
\[
f'(u_r(t))<f'(u_l(t)),\qquad t\geq0\,.
\]
Of course this requires initially $f'(u_r(0))<f'(u_l(0))$. Recall however that discontinuities with $f'(u_r(0))>f'(u_l(0))$ are already covered by Corollary~\ref{C.constant}.

\begin{Proposition}\label{P.small-shock}
Let $f,\,g\in \cC^2(\RR)$ and $\uu\in\RR$ be such that
\[g(\uu)=0\,,\qquad g'(\uu)<0\qquad\textrm{and}\qquad f''(\uu)\neq0\,.\]
For any $C_0>1$, there exist $\epsilon>0$ and $C>0$ such that for any $\psi_0\in\RR$ and $\tilde v_0\in BUC^1(\RR^\star)$ satisfying
\begin{equation}\label{init-condition-small-shock}
 \Norm{\tilde v_0}_{L^\infty(\RR^\star)}\leq\epsilon \quad \text{ and } \quad \Norm{(\sign(f''(\uu))\,\d_x\tv_0)_-}_{L^{\infty}(\RR^\star)}\leq\epsilon\,,
\end{equation}
and
\[
\lim_{\delta\searrow 0} f'(\uu+\tilde v_0(\delta))<\lim_{\delta\searrow 0} f'(\uu+\tilde v_0(-\delta)),
\]
there exists $\psi\in\cC^2(\RR^+)$ and $u\in BUC^1(\Omega^\psi)$ with initial data $\psi(0)=\psi_0$ and $u(0,\cdot)=(\uu+\tv_0)(\cdot+\psi_0)$ such that $u$ is an entropy solution to~\eqref{eq-u} and satisfies for any $t\geq 0$ 
\begin{align*}
\Norm{ u(t,\cdot-\psi(t))-\uu}_{L^{\infty}(\RR\setminus\{\psi(t)\})}&\leq \Norm{\tilde v_0}_{L^{\infty}(\RR^\star)} C_0\,e^{g'(\uu)\,t}\,,\\
\Norm{(\sign(f''(\uu))\,\d_x  u(t,\cdot-\psi(t)))_-}_{L^{\infty}(\RR\setminus\{\psi(t)\})}&\leq \Norm{(\sign(f''(\uu))\,\d_x\tilde v_0)_-}_{L^{\infty}(\RR^\star)} C_0\,e^{g'(\uu)\,t}\,,\\
\Norm{\d_x  u(t,\cdot-\psi(t))}_{L^{\infty}(\RR\setminus\{\psi(t)\})}&\leq \Norm{\d_x \tilde v_0}_{L^{\infty}(\RR^\star)} C_0\,e^{g'(\uu)\,t}\,,\\
\abs{\psi'(t)-f'(\uu)}&\leq \Norm{\tilde v_0}_{L^{\infty}(\RR^\star)} C\, e^{g'(\uu)\, t}\,,
\end{align*}
and moreover there exists $\psi_\infty$ such that 
\[
\abs{\psi_\infty-\psi_0}\,\leq \Norm{\tilde v_0}_{L^{\infty}(\RR^\star)} C\,,
\]
and for any $t\geq 0$
\[
\abs{\psi(t)-\psi_\infty-t\,f'(\uu)}\,\leq \Norm{\tilde v_0}_{L^{\infty}(\RR^\star)} C\, e^{g'(\uu)\, t}\,.
\]
\end{Proposition}

\begin{Remark}\label{Rk.gradient}
The consideration of perturbation by small shocks is partly motivated by the fact that smooth perturbations, small in $L^\infty$ but not in $W^{1,\infty}$, may indeed form shocks in finite time. Note however that whereas Proposition~\ref{P.constant-conditional} does follow smooth solutions until shock formation, Proposition~\ref{P.small-shock} cannot be used right after shock formation since it requires (asymmetric) smallness of the smooth part of the gradient. Indeed Proposition~\ref{P.small-shock} is a counterpart to Proposition~\ref{P.constant} whereas an analog to Proposition~\ref{P.constant-conditional} would be more appropriate near a shock formation. Note however that then the ``smooth" part of solutions would then be controlled only in $W^{1,1}$.
\end{Remark}

\begin{Remark}
Since the problem is known to be globally well-posed in the class of $L^\infty$ entropy solutions, one may rightfully wonder whether the result could be extended to such a general class. Such an extension would lead us a way beyond the scope of the present contribution, focused on piece-wise smooth solutions, and very close to front-tracking algorithms. Without going that far, let us now give some hints about first steps required to extend our strategy in this direction. Note first that it is straightforward to extend Proposition~\ref{P.small-shock} to cases when the initial data contains discontinuities leading to rarefaction waves and an arbitrary number of well-separated shocks. Going beyond the latter case to allow for interacting shocks seems a more tedious task but seemingly still achievable with arguments in the spirit of those expounded in the present contribution. In particular even in the latter case one expects that no new discontinuity arises and that paths of discontinuities could be predicted by linearized dynamics. However to relax constraints on derivatives, one would need to follow the path sketched in Remark~\ref{Rk.gradient} or to approximate $L^\infty$ initial data by piece-wise smooth initial data containing only flat or almost flat smooth parts but an arbitrary large number of shocks. In both cases the prediction of the regularity structure would be a much harder task.
\end{Remark}

\begin{proof}[Proof of Proposition~\ref{P.small-shock}] 
To spare notational complexity, we assume henceforth that $\psi_0=0$ and accordingly drop tildes on $\tv_0$. The general case may be dealt with either by using translation invariance or by propagating notational changes.

We recall that the proof strategy is the following. Given an initial data $v_0$ satisfying~\eqref{init-condition-small-shock}, we define two extensions $v_{0,\pm}$, defined on $\RR$, satisfying $v_{0,\pm}=v_0$ on $ \RR^\pm$, and fulfilling the hypotheses of Proposition~\ref{P.constant} near $\uu$. Consider $u_\pm$ the two global unique classical solutions to~\eqref{eq-u}  emerging from the initial data $u_\pm\id{t=0}=\uu +v_{0,\pm}$. The solution $u$ is constructed by patching together $u_+$ and $u_-$ along the curve $\psi(t)$ defined through the Rankine-Hugoniot condition.

The first step is carried out thanks to the following Lemma.
\begin{Lemma}\label{L.multivalued}
For any $C_0^{(0)}>1$ and any $v_0\in BUC^1(\RR^\star)$ there exist
$v_{0,\pm}\in BUC^1(\RR)$ satisfying
\[v_0(x)=\begin{cases} 
v_{0,+}(x)&\text{ if } x>0,\\
v_{0,-}(x) &\text{ if } x<0,
\end{cases}
\]
and
\begin{align*}
\Norm{v_{0,\pm}}_{L^{\infty}(\RR)}&\leq \Norm{v_0}_{L^{\infty}(\RR^\pm)} C_0^{(0)}\,,\\
\Norm{(\d_x v_{0,\pm})_-}_{L^{\infty}(\RR)}&\leq \Norm{(\d_xv_0)_-}_{L^{\infty}(\RR^\pm)}\,,\\
\Norm{(-\d_x v_{0,\pm})_-}_{L^{\infty}(\RR)}&\leq \Norm{(-\d_xv_0)_-}_{L^{\infty}(\RR^\pm)}\,,\\
\Norm{\d_x v_{0,\pm}}_{L^{\infty}(\RR)}&\leq \Norm{\d_x v_0}_{L^{\infty}(\RR^\pm)}\,.
\end{align*}
\end{Lemma}

\begin{proof} Since the situation is symmetric we only show how to extend the right part of $v_0$. To do so let us introduce 
\[
v_0(0^+)\eqdef \lim_{x\searrow 0 } v_0(x)\qquad\textrm{and}\qquad
\d_xv_0(0^+)\eqdef \lim_{x\searrow 0 } \d_xv_0(x)
\]
whose existence is guaranteed by uniform continuity.

We set 
\[\delta=
\frac{2(C_0^{(0)}-1)\Norm{v_0}_{L^\infty(\RR^+)}}{\max(\{1,|\d_xv_0(0^+)|\})}
\]
and define
\[ v_{0,+}(x)=\begin{cases}
v_0(0^+)-\frac\delta2\d_xv_0(0^+)&\text{ if } x\in(-\infty,-\delta]\,,\\
v_0(0^+)+  \big(x+\frac12\delta^{-1}x^2\big)\d_xv_0(0^+)&\text{ if } x\in(-\delta,0]\,,\\
v_0(x) &\text{ if } x>0\,.
\end{cases}
\]
One readily checks that $v_{0,+}$ satisfies all prescribed constraints.
\end{proof}

We can now proceed with the proof of Proposition~\ref{P.small-shock}. We denote $C_0$ the prescribed amplifying constant and $\epsilon$ the smallness parameter as in the statement. First we apply Lemma~\ref{L.multivalued} with amplification constant $C_0^{(0)}=\sqrt{C_0}$ to receive extensions $v_{0,\pm}$. Then we apply twice Proposition~\ref{P.constant}, with initial perturbations $v_{0,\pm}$ near $\uu$ and prescribed amplification factors $C_0^\pm=\sqrt{C_0}$. This is licit provided we constrain $\epsilon$ by
\[
\sqrt{C_0}\epsilon \leq \epsilon_0
\]
where $\epsilon_0$ encodes the smallness constraint arising from Proposition~\ref{P.constant}. Hence the existence of $u_\pm\in BUC^1(\RR^+\times\RR)$ global unique classical solutions to~\eqref{eq-u} with initial data $u_\pm\id{t=0}=\uu+v_{0,\pm}$ satisfying 
for any $t\geq0$,
\begin{align*}
\Norm{u_\pm(t,\cdot)-\uu}_{L^{\infty}(\RR)}&\leq \Norm{v_0}_{L^{\infty}(\RR^\star)} C_0\,e^{g'(\uu)\,t}\,,\\
\Norm{(\sign(f''(\uu))\,\d_x u_\pm(t,\cdot))_-}_{L^{\infty}(\RR)}&\leq \Norm{(\sign(f''(\uu))\,\d_xv_0)_-}_{L^{\infty}(\RR^\star)} C_0\,e^{g'(\uu)\,t}\,,\\
\Norm{\d_x u_\pm(t,\cdot)}_{L^{\infty}(\RR)}&\leq \Norm{\d_x v_0}_{L^{\infty}(\RR^\star)} C_0\,e^{g'(\uu)\,t}\,.
\end{align*}

We shall construct our solution, $u$, through the following formula:
\begin{equation}\label{patch}
u(t,x)=
\begin{cases} 
u_-(t,x) & \text{ if } x<\psi(t),\\
u_+(t,x) & \text{ if } x>\psi(t),
\end{cases}
\end{equation}
where the discontinuity curve, $\psi$, is defined through the Rankine-Hugoniot condition
\[
\big(u_+(t,\psi(t))-u_-(t,\psi(t))\big)\psi'(t)\,=\, f(u_+(t,\psi(t)))-f(u_-(t,\psi(t)))\,.
\]

To this aim, we introduce the slope function associated with $f$:
\begin{equation}\label{slope-function}
s_f:\,\RR\times\RR\to\RR,\quad (a,b)\mapsto \int_0^1 f'\big(a+\tau(b-a)\big)\,\dd \tau.
\end{equation}
We have $s_f\in \cC^1(\RR\times\RR)$. In particular, $(t,x)\mapsto s_f(u_-(t,x),u_+(t,x))\in BUC^1(\RR_+\times\RR)$, hence there exists a unique $\psi\in\cC^2(\RR_+)$ satisfying $\psi(0)=0$ and for any $t\geq 0$,
\[\psi'(t)\,=\, s_f(u_-(t,\psi(t)),u_+(t,\psi(t)))\,.\]
It follows that $\psi$ satisfies the Rankine-Hugoniot condition as well as the claimed estimates. Indeed, we have for any $t\geq 0$
\[\psi'(t)-f'(\uu)\,=\,s_f(u_-(t,\psi(t)),u_+(t,\psi(t)))-s_f(\uu,\uu)\]
and since $s_f$ is a locally Lipschitz function, the bound on $\psi'(t)-f'(\uu)$ stems directly from the known bounds on $\Norm{u_\pm(t,\cdot)-\uu}_{L^\infty(\RR)}$. Now the bound on $\psi'-f'(\uu)$ may be integrated to conclude the desired estimate with
\[
\psi_\infty=\int_0^\infty (\psi'(t)-f'(\uu))\,\dd t\,.
\]

To achieve the proof, we need to ensure that lessening $\epsilon$ if necessary, the constructed weak solution is an entropy solution. Since $f''(\uu)\neq 0$, we can restrict $\epsilon$ so that $f$ is either strictly concave or strictly convex on $[\uu-C_0\epsilon,\uu+C_0\epsilon]$ and hence $u$ is an entropy solution if and only if Lax's condition holds, {\em i.e.}
\[
f'(u_+(t,\psi(t)))<f'(u_-(t,\psi(t))),\qquad t\geq0\,.
\]
Since the corresponding inequality holds at time $t=0$ and $f'$ is one-to-one on $[\uu-C_0\epsilon,\uu+C_0\epsilon]$, it is sufficient to prove that
\[w(t):=u_+(t,\psi(t))- u_-(t,\psi(t))\neq 0,\qquad t>0\,.\]
Notice
\begin{align*}
w'(t) &=\left(\d_t u_+ +\psi'(t)\d_x u_+-\d_t u_- -\psi'(t)\d_x u_-\right)(t,\psi(t))\\
&=\Big(g(u_+)-g(u_-)+\big(\psi'(t)-f'(u_+)\big)\d_x u_+-\big(\psi'(t)-f'(u_-)\big)\d_x u_-\Big)(t,\psi(t))\\
&=\Phi(t,w(t))
\end{align*}
with
\begin{align*}
\Phi\,:\,(t,z)\mapsto&\ s_g(u_+(t,\psi(t)),u_-(t,\psi(t)))\,z\\
&+\Big(s_f(u_+(t,\psi(t)),u_+(t,\psi(t))-z)-s_f(u_+(t,\psi(t)),u_+(t,\psi(t)))\Big)\d_x u_+(t,\psi(t))\\
&-\Big(s_f(u_-(t,\psi(t))+z,u_-(t,\psi(t)))-s_f(u_-(t,\psi(t)),u_-(t,\psi(t)))\Big)\d_x u_-(t,\psi(t)).
\end{align*}
Since $\Phi$ is $\cC^1$ and ($\forall t\geq0$, $\Phi(t,0)=0$), an application of the Cauchy-Lipschitz theorem concludes the proof.
\end{proof}

\section{Asymptotic stability of shocks}

\subsection{Asymptotic stability under shockless perturbations}\label{s:shockless}

In this section under natural spectral assumptions we show the asymptotic stability under regular perturbations of entropy-admissible Riemann shocks of~\eqref{eq-u}. More precisely, as described in the introduction we consider a uniformly traveling wave $\uu$, 
\[\uu(t,x)=\uU(x-(\psi_0+\sigma t))\,,\] 
with initial shock position $\psi_0\in\RR$, speed $\sigma\in\RR$ and wave profile $\uU$
\begin{equation}\label{uU} \uU(x)=\begin{cases}
\uu_- &\text{ if }  x<0\\
\uu_+ &\text{ if }  x>0
\end{cases}\end{equation}
where $(\uu_-,\uu_+)\in\RR^2$, $\uu_+\neq\uu_-$. The problem is invariant by translation and $\psi_0$ is arbitrary, whereas speed and profile are assumed to satisfy conditions enforcing that $\uu$ is a stable entropy solution.
To ensure that $\uu$ is a weak solution, we require that $(\sigma,\uu_-,\uu_+)$ satisfies the equilibrium condition
\begin{equation}\label{equilibrium-shock}
g(\uu_+)=0 \quad \text{ and } \quad g(\uu_-)=0\,;
\end{equation}
and the Rankine-Hugoniot condition 
\begin{equation}\label{RH-shock}
f(\uu_+)-f(\uu_-)=\sigma (\uu_+-\uu_-)\,.
\end{equation}
(Strict) entropy admissibility then amounts to the following Oleinik condition
\begin{equation}\label{gLax-shock}
\begin{cases}
\qquad\qquad\qquad\sigma\,>\,f'(\uu_+)\,,&\\[0.5em]
\frac{f(\tau\,\uu_-+(1-\tau)\,\uu_+)-f(\uu_-)}{\tau\,\uu_-+(1-\tau)\,\uu_+-\uu_-}>
\frac{f(\tau\,\uu_-+(1-\tau)\,\uu_+)-f(\uu_+)}{\tau\,\uu_-+(1-\tau)\,\uu_+-\uu_+}&\qquad\textrm{for any }\ \tau\in(0,1)\,,\\[0.5em]
\qquad\qquad\qquad f'(\uu_-)\,>\,\sigma\,,
\end{cases}
\end{equation}
and the spectral stability is encoded in 
\begin{equation}\label{stab-spec-shock}
g'(\uu_+)<0 \quad \text{ and } \quad g'(\uu_-)<0\,.
\end{equation}
\begin{Remark}\label{R.Lax-vs-Oleinik}
Note that the entropy condition~\eqref{gLax-shock} contributes to the stability properties of the shock, in a more subtle way than~\eqref{stab-spec-shock}. To begin, let us point out that without assuming condition~\eqref{gLax-shock} the proof of Theorem~\ref{Th.shock} below provides a weak solution to~\eqref{eq-u}. Yet in full generality the solution depends then on the choices of initial extensions. This lack of uniqueness is even captured by a direct spectral analysis, as hinted at in Section~\ref{s:conclusion}, and the corresponding linearized dynamics is ill-posed. Replacing~\eqref{gLax-shock} with the weaker Lax condition
\[
f'(\uu_+)<\sigma < f'(\uu_-)
\] 
restores spectral stability and ensures uniqueness in a suitable class of piecewise smooth solutions. The uniqueness in the class of entropy-admissible solutions if Oleinik's condition is satisfied is provided by the theory due to Kru\v zkov~\cite{Kruzhkov}. The full condition~\eqref{gLax-shock} would also be crucial to the stability properties of the shock if one allowed perturbations breaking the large shock into a ``sum" of smaller subshocks. See~\cite[Remark~4.7]{Bressan} for a more detailed discussion and more generally~\cite[Chapters~4 and~6]{Bressan} for classical background on entropy solutions.
\end{Remark}

As in Section~\ref{s:constant_by_shock} we shall solve~\eqref{eq-u} in the class of piecewise regular functions and adopt conventions introduced there. The main difference is that now we require as entropy condition, for any $t\geq0$
\begin{equation}\label{gLax}
\begin{cases}
\qquad\qquad\qquad\psi'(t)\,>\,f'(u_r(t))\,,&\\[0.5em]
\frac{f(\tau\,u_l(t)+(1-\tau)\,u_r(t))-f(u_l(t))}{\tau\,u_l(t)+(1-\tau)\,u_r(t)-u_l(t)}>
\frac{f(\tau\,u_l(t)+(1-\tau)\,u_r(t))-f(u_r(t))}{\tau\,u_l(t)+(1-\tau)\,u_r(t)-u_r(t)}&\qquad\textrm{for any }\ \tau\in(0,1)\,,\\[0.5em]
\qquad\qquad\qquad f'(u_l(t))\,>\,\psi'(t)\,.
\end{cases}
\end{equation}
where $u_l(t)=\lim_{\delta\searrow 0} u(t,\psi(t)-\delta)$ and $u_r(t)=\lim_{\delta\searrow 0} u(t,\psi(t)+\delta)$.
\begin{Theorem}\label{Th.shock}
Let $f,g\in \mathcal{C}^2(\RR)$ and $(\sigma,\uu_-,\uu_+)\in\RR^3$ satisfying~\eqref{equilibrium-shock}-\eqref{RH-shock}-\eqref{gLax-shock}-\eqref{stab-spec-shock} and
\begin{equation} \label{hyp-convex-shock}
f''(\uu_+)\neq 0\quad \text{ and }\quad  f''(\uu_-)\neq 0\,.
\end{equation}
For any $C_0>1$, there exists $\epsilon>0$ and $C>0$ such that for any $\psi_0\in\RR$ and $\tv_0\in BUC^1(\RR^\star)$ satisfying
\begin{equation}\label{init-condition-shock}
\begin{array}{rl}
\Norm{\tv_0}_{L^\infty(\RR^\star)}&\leq\epsilon\,,\\
\Norm{(\sign(f''(\uu_+))\,\d_x\tv_0)_-}_{L^{\infty}(\RR^+)}&\leq\epsilon\,,\\
\Norm{(\sign(f''(\uu_-))\,\d_x\tv_0)_-}_{L^{\infty}(\RR^-)}&\leq\epsilon\,,
\end{array}
\end{equation}
there exists $\psi\in\cC^2(\RR^+)$ with initial data $\psi(0)=\psi_0$ such that the entropy solution to~\eqref{eq-u}, $u$, generated by the initial data $u(0,\cdot)=(\uU+\tv_0)(\cdot+\psi_0)$, $\uU$ being as in \eqref{uU}, belongs to $ BUC^1(\Omega^\psi)$ and satisfies for any $t\geq 0$
\begin{align*}
\Norm{u(t,\cdot-\psi(t))-\uu_\pm}_{L^{\infty}(\RR^\pm)}&\leq \Norm{\tv_0}_{L^{\infty}(\RR^\pm)} C_0\,e^{g'(\uu_\pm)\,t}\,,\\
\Norm{(\sign(f''(\uu_\pm))\,\d_x u(t,\cdot-\psi(t)))_-}_{L^{\infty}(\RR^\pm)}&\leq \Norm{(\sign(f''(\uu_\pm))\,\d_x\tv_0)_-}_{L^{\infty}(\RR^\pm)} C_0\,e^{g'(\uu_\pm)\,t}\,,\\
\Norm{\d_x u(t,\cdot-\psi(t))}_{L^{\infty}(\RR^\pm)}&\leq \Norm{\d_x \tv_0}_{L^{\infty}(\RR^\pm)} C_0\,e^{g'(\uu_\pm)\,t}\,,\\
\abs{\psi'(t)-\sigma}&\leq \Norm{\tv_0}_{L^{\infty}(\RR^\star)}\,C\, e^{\max(\{g'(\uu_+),g'(\uu_-)\})\, t}\,,
\end{align*}
and moreover there exists $\psi_\infty$ such that 
\[
\abs{\psi_\infty-\psi_0}\,\leq \Norm{\tv_0}_{L^{\infty}(\RR^\star)} C\,,
\]
and for any $t\geq 0$
\[
\abs{\psi(t)-\psi_\infty-t\,\sigma}\,\leq \Norm{\tv_0}_{L^{\infty}(\RR^\star)} C\, e^{\max(\{g'(\uu_+),g'(\uu_-)\})\, t}\,.
\]
\end{Theorem}

\begin{Remark}\label{Rk.constant-shock}
Theorem~\ref{Th.shock} is a direct counterpart to Proposition~\ref{P.constant}. We could also derive from it an analogous to Corollary~\ref{C.constant}. Likewise as in Proposition~\ref{P.constant-classical} we could relax totally or partly hypothesis~\eqref{hyp-convex-shock} if~\eqref{init-condition-shock} is strengthened. This would lead to four different versions of Theorem~\ref{Th.shock}. We could also provide a counterpart to Proposition~\ref{P.constant-conditional}.

Modifications required to prove the foregoing claims are straightforward and we have chosen to omit them so as to avoid redundancy.
\end{Remark}

\begin{Remark}
Note that expressed in classical stability terminology (see for instance~\cite{Henry-geometric}) we have proved orbital stability with asymptotic phase. We stress however that the role of phase shifts is here deeper than in the classical stability analysis of smooth waves since it is not only required to provide decay of suitable norms in large-time but also to ensure that these norms are finite locally in time. In particular here there is no freedom, even in finite time, in the definition of phase shifts. See~\cite[Section~4.1]{JNRYZ} and~\cite{DR2} for related (more elaborate) discussions.
\end{Remark}

\begin{proof}[Proof of Theorem~\ref{Th.shock}.] 
The proof of Theorem~\ref{Th.shock} follows closely the construction given in the proof of Proposition~\ref{P.small-shock}. We also assume henceforth that $\psi_0=0$, without loss of generality. Using Lemma~\ref{L.multivalued} and Proposition~\ref{P.constant}, we find that for $\epsilon>0$ sufficiently small and for any $v_0\in BUC^1(\RR^\star)$ satisfying~\eqref{init-condition-shock},
there exist $u_\pm\in BUC^1(\RR^+\times\RR)$ global classical solutions to~\eqref{eq-u}  with initial data $u_\pm\id{t=0}=\uu_\pm+v_{0,\pm}$ and satisfying the desired estimates.
We can now construct the solution, $u$, through~\eqref{patch} where $\psi$ is defined by the differential equation
\[\psi'(t)\,=\, s_f(u_-(t,\psi(t)),u_+(t,\psi(t)))\,,\]
where $s_f$ is defined in~\eqref{slope-function}, so that the Rankine-Hugoniot condition as well as the desired bounds on $\psi$ hold since for any $t\geq 0$,
\[\psi'(t)-\sigma\,=\,s_f(u_-(t,\psi(t)),u_+(t,\psi(t)))-s_f(\uu_-,\uu_+)\,.\]
Then the last estimates on $\psi$ are obtained by integration with 
\[
\psi_\infty=\int_0^\infty (\psi'(t)-\sigma)\,\dd t\,.
\]

To achieve the proof of Theorem~\ref{Th.shock} we only need to ensure that by lessening $\epsilon$ further if necessary formula~\eqref{patch} ensures~\eqref{gLax}. For this purpose we consider 
\[
S_f:\,\RR\times\RR\times[0,1]\to\RR,\qquad
(a,b,\tau)\mapsto
s_f(a,\tau\,a+(1-\tau)\,b)
-s_f(b,\tau\,a+(1-\tau)\,b)
\]
and observe that it is continuous. Since $\{\uu_-\}\times\{\uu_+\}\times[0,1]$ is compact and for any $\tau\in[0,1]$, $S_f(\uu_-,\uu_+,\tau)>0$, one may ensure that provided $\epsilon$ is small enough, for any $(a,b)$ such that $|a-\uu_-|\leq C_0\epsilon$ and $|b-\uu_+|\leq C_0\epsilon$, and any $\tau\in[0,1]$, $S_f(a,b,\tau)>0$. From this stems~\eqref{gLax} for $u$ built from~\eqref{patch}, and the proof is complete.
\end{proof}

We now prove that the exponential decay of higher derivatives holds provided we assume the stronger (symmetric) smallness condition on the first derivative.
\begin{Proposition}\label{P.shock-higher-derivatives}
Let $k\in\NN$, $k\geq2$, $f\in\cC^{k+1}(\RR),\ g\in \cC^k(\RR)$ and $(\sigma,\uu_-,\uu_+)\in\RR^3$ satisfying~\eqref{equilibrium-shock}-\eqref{RH-shock}-\eqref{gLax-shock}-\eqref{stab-spec-shock}. There exists $\epsilon>0$ and $C_k$ such that for any $\psi_0\in\RR$ and $\tv_0\in BUC^k(\RR^\star)$ satisfying
\begin{equation}\label{init-condition-shock-isotropic}
\Norm{\tv_0}_{L^\infty(\RR^\star)}\leq\epsilon \quad  \text{ and } \quad 
\Norm{\d_x\tv_0}_{L^{\infty}(\RR^\star)}\leq\epsilon\,,
\end{equation}
there exist $\psi\in\cC^{k+1}(\RR^+)$ and $u\in BUC^k(\Omega^\psi)$ with initial data $\psi(0)=\psi_0$ and $u(0,\cdot)=(\uU+\tv_0)(\cdot+\psi_0)$ such that $u$ is an entropy solution to~\eqref{eq-u} and satisfies for any $t\geq 0$ and any $j\in \{0,\dots,k\}$
\begin{align*}
\Norm{u(t,\cdot-\psi(t))-\uu_\pm}_{W^{j,\infty}(\RR^\pm)}&\leq \Norm{\tv_0}_{W^{j,\infty}(\RR^\pm)} C_k\,e^{g'(\uu_\pm)\,t}\,,\\
\abs{\psi^{(j+1)}(t)}&\leq \Norm{\tv_0}_{W^{j,\infty}(\RR^\star)}\,C_k\, e^{\max(\{g'(\uu_+),g'(\uu_-)\})\, t}\,.
\end{align*}
\end{Proposition}
\begin{proof}
The result does not follow directly from Proposition~\ref{P.high-order} applied to $u_\pm$ defined in the proof of Theorem~\ref{Th.shock}, because the initial data provided by Lemma~\ref{L.multivalued} is not sufficiently regular. Here we rather rely on the following Lemma deduced from a standard extension theorem~\cite[Theorem~4.26]{Adams}.
\begin{Lemma}\label{L.multivalued-high-order}
Let $k\in\NN$, $k\geq 2$. There exists $C_k>0$ such that for any $v_0\in BUC^k(\RR^\star)$, there exist
$v_{0,\pm}\in \cC^k(\RR)$ satisfying
\[v_0(x)=\begin{cases} 
v_{0,+}(x)&\text{ if } x>0,\\
v_{0,-}(x) &\text{ if } x<0,
\end{cases}
\]
and for any $j\in\NN$, $0\leq j\leq k$,
\begin{equation}\label{control-linear}
\Norm{\d_x^j v_{0,\pm}}_{L^\infty(\RR)}\leq  \Norm{\d_x^j v_0}_{L^\infty(\RR^\pm)}C_k\,.
\end{equation}
\end{Lemma}
Replacing Lemma~\ref{L.multivalued} and Proposition~\ref{P.constant} with Lemma~\ref{L.multivalued-high-order} and~Proposition~\ref{P.high-order}, the proof of Proposition~\ref{P.shock-higher-derivatives} is then almost identical to the proof of Theorem~\ref{Th.shock}. 
\end{proof}

We can also obtain a counterpart to Proposition~\ref{P.shock-higher-derivatives} with the asymmetric smallness assumption on first-order derivatives.
\begin{Proposition}\label{P.shock-higher-derivatives-2}
Let $k\in\NN$, $k\geq2$, $f\in\cC^{k+1}(\RR),\ g\in \cC^k(\RR)$ and $(\sigma,\uu_-,\uu_+)\in\RR^3$ satisfying~\eqref{equilibrium-shock}-\eqref{RH-shock}-\eqref{gLax-shock}-\eqref{stab-spec-shock} and~\eqref{hyp-convex-shock}. There exist $\epsilon>0$ and $C_k>0$ such that for any $\psi_0\in\RR$ and $\tv_0\in BUC^k(\RR^\star)$ satisfying~\eqref{init-condition-shock},
the entropy solution defined in Theorem~\ref{Th.shock} satisfies $u\in BUC^k(\Omega^\psi)$, $\psi\in\cC^{k+1}(\RR^+)$ and for any $t\geq 0$ and any $j\in \{1,\dots,k\}$
\begin{align*}
\big\lVert\d_x^j u(t,\cdot&-\psi(t))\big\rVert_{L^\infty(\RR^\pm)}\\
&\leq 
\Norm{\tv_0}_{W^{j,\infty}(\RR^\pm)}C_k
(1+\Norm{\d_x^2\tv_0}_{L^\infty(\RR^\pm)}^{j-1})
e^{C_k\,\|\tv_0\|_{W^{1,\infty}(\RR^\pm)}\,(1+\|\tv_0\|_{W^{1,\infty}(\RR^\pm)}^{j-1})}\,
\, 
\,e^{g'(\uu_\pm)\,t}\,,\\
\abs{\psi^{(j+1)}(t)}
&\leq \Norm{\tv_0}_{W^{j,\infty}(\RR^\star)}C_k
(1+\Norm{\d_x^2\tv_0}_{L^\infty(\RR^\star)}^{j-1})
e^{C_k\,\|\tv_0\|_{W^{1,\infty}(\RR^\star)}\,(1+\|\tv_0\|_{W^{1,\infty}(\RR^\star)}^{j-1})}\,
\, e^{\max(\{g'(\uu_+),g'(\uu_-)\})\, t}\,.
\end{align*}
\end{Proposition}
\begin{proof}
Although we follow the same strategy as in the earlier results, we need to ensure that the regular extensions $\tv_{0,\pm}\in BUC^k(\RR) $ preserve the asymmetric smallness hypothesis~\eqref{init-condition-shock}. To this aim, we introduce a smooth cut-off function, $\chi$, such that $\chi(x)=0$ for $|x|\geq 2/3$, $\chi(x)=1$ for $|x|\leq 1/3$ and $\chi(x)\in[0,1]$ for $x\in\RR$, and define
\[ \tv_{0,+}(x)=\begin{cases}
\tv_0(0^+)+ \int_0^{-\delta} w_{0,+}(y)\chi(\delta^{-1}y) \dd y&\text{ if } x\in(-\infty,-\delta]\,,\\
\tv_0(0^+)+ \int_0^x w_{0,+}(y)\chi(\delta^{-1}y) \dd y&\text{ if } x\in(-\delta,0]\,,\\
\tv_0(x) &\text{ if } x>0\,.
\end{cases}
\]
where $w_{0,+}$ is the extension associated with $\d_x\tv_0$ provided by Lemma~\ref{L.multivalued-high-order}. When choosing 
\[\delta=\min\left(\left\{\frac{c_0\,\epsilon}{1+\Norm{\d_x\tv_0}_{W^{1,\infty}(\RR^+)}},1\right\}\right),\]
with $c_0>0$ sufficiently small and defining symmetrically $\tv_{0,-}$, we derive the following Lemma.
\begin{Lemma}\label{L.multivalued-high-order-anisotropic}
Let $k\in\NN$, $k\geq 2$, $C_0>1$ and $\epsilon>0$. There exists $C_k>0$ such that for any $\tv_0\in BUC^k(\RR^\star)$ satisfying~\eqref{init-condition-shock}, there exist
$\tv_{0,\pm}\in BUC^k(\RR)$ satisfying
\[\tv_0(x)=\begin{cases} 
\tv_{0,+}(x)&\text{ if } x>0,\\
\tv_{0,-}(x) &\text{ if } x<0,
\end{cases}
\]
and the estimates
\[
\begin{array}{rl}
\Norm{\tv_{0,\pm}}_{L^\infty(\RR)}&\leq \min(\{C_0\epsilon,\Norm{\tv_0}_{W^{1,\infty}(\RR^\star)}C_k\})\,,\\
\Norm{(\sign(f''(\uu_+))\,\d_x\tv_{0,+})_-}_{L^{\infty}(\RR)}&\leq C_0\epsilon\,,\\
\Norm{(\sign(f''(\uu_-))\,\d_x\tv_{0,-})_-}_{L^{\infty}(\RR)}&\leq C_0\epsilon\,,
\end{array}
\]
and for any $j\in\NN$, $1\leq j\leq k$,
\[
\Norm{\tv_{0,\pm}}_{W^{j,\infty}(\RR)}\leq  \Norm{\tv_0}_{W^{j,\infty}(\RR^\pm)}(1+\Norm{\d_x\tv_0}_{W^{1,\infty}(\RR^\pm)}^{j-1})C_k\,.
\]
\end{Lemma}
We can now follow the proof of Theorem~\ref{Th.shock}, replacing Lemma~\ref{L.multivalued} and Proposition~\ref{P.constant} with Lemma~\ref{L.multivalued-high-order-anisotropic} and~Proposition~\ref{P.high-order}. .
\end{proof}

\begin{Remark}
The non-uniqueness of the intermediate stage of our proofs is particularly striking here, since even for the same initial data, depending on the level of regularity we aim at, we build distinct extended solutions. Yet in the end, as discussed in Remark~\ref{R.Lax-vs-Oleinik}, the parts actually used in the final gluing process are indeed independent of choices in the extension as a consequence of Lax's condition, and uniqueness holds by the theory of Kru\v zkov~\cite{Kruzhkov}.
\end{Remark}

\subsection{Perturbation by small shocks}\label{s:large_and_small_shocks}

We now elaborate on Proposition~\ref{P.small-shock}  and Theorem~\ref{Th.shock} and perturb a spectrally stable strictly entropy-admissible Riemann shock of~\eqref{eq-u} with a perturbation containing one shock. For concreteness and concision we assume that the small shock is located on the left of the large shock, the opposite situation being deduced by symmetry considerations. Since the Riemann shock is strictly entropy-admissible, sufficiently small perturbations with a small shock will produce two paths of discontinuity eventually merging in a single one, the small shock being essentially absorbed by the large one. 

We follow the position of the large shock with $\psi:\ \RR\to\RR$ and the position of the small shock, as long as it persists, with $\psi_s:\ [0,t^\star]\to\RR$ where $t^\star>0$, $\psi_s(t^\star)=\psi(t^\star)$ and, for any $t\in[0,t^\star)$, $\psi_s(t)<\psi(t)$. In particular we seek for a solution that is a classical solution on the domain
\[
\Omega_{\psi,\psi_s}\eqdef\RR_+\times\RR\setminus \Big(\{\,(t,\psi_s(t))\,|\,t\in[0,t^\star]\,\}\cup\{\,(t,\psi(t))\,|\,t\geq 0\}\Big)\,
\]
(see Figure~\ref{fig:sketch_shocks}).

\begin{figure}[tbph]
\subfigure[Original domain]{
\includegraphics[width=.5\textwidth]{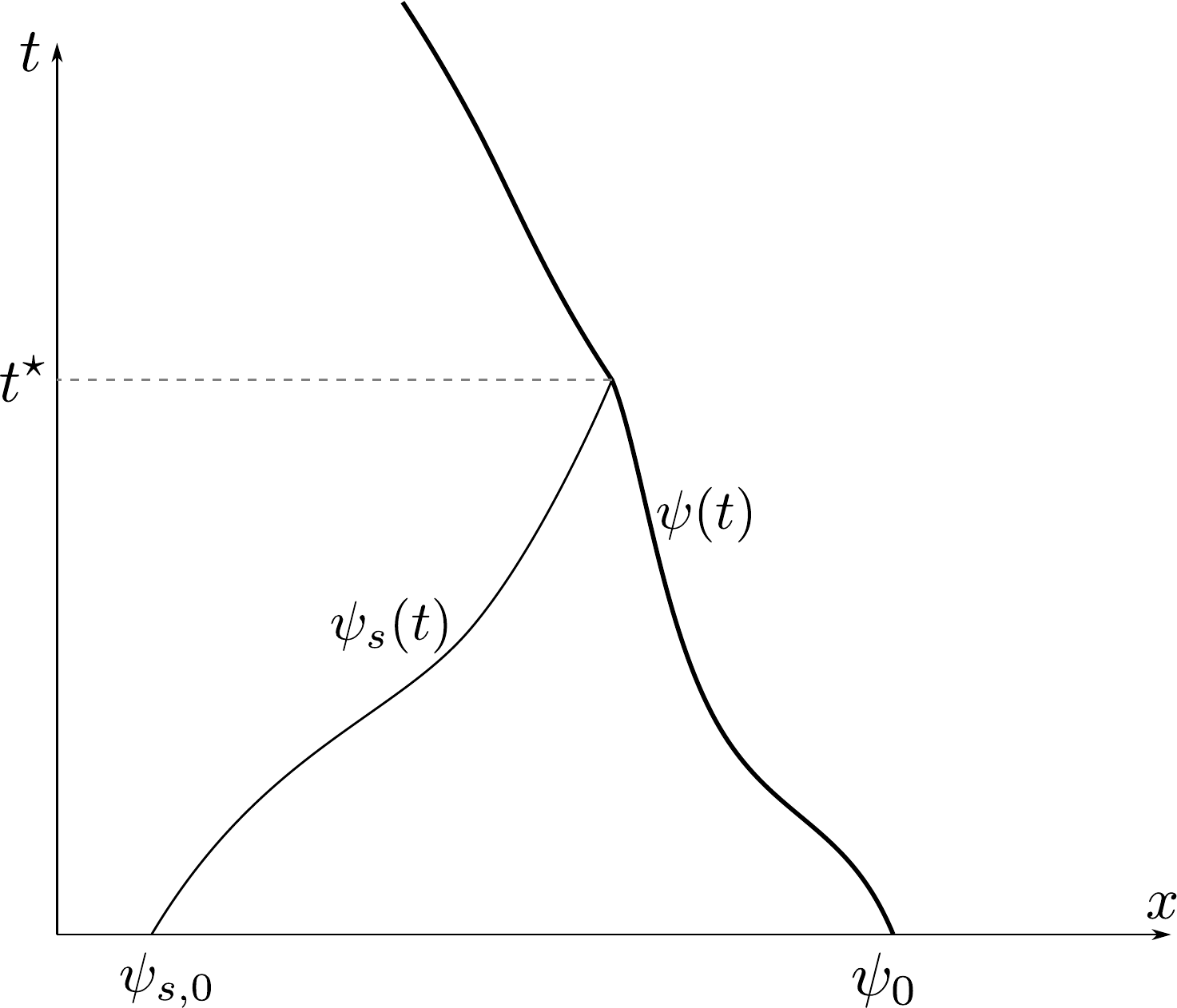}
}
\subfigure[Straightened domain]{
\includegraphics[width=.5\textwidth]{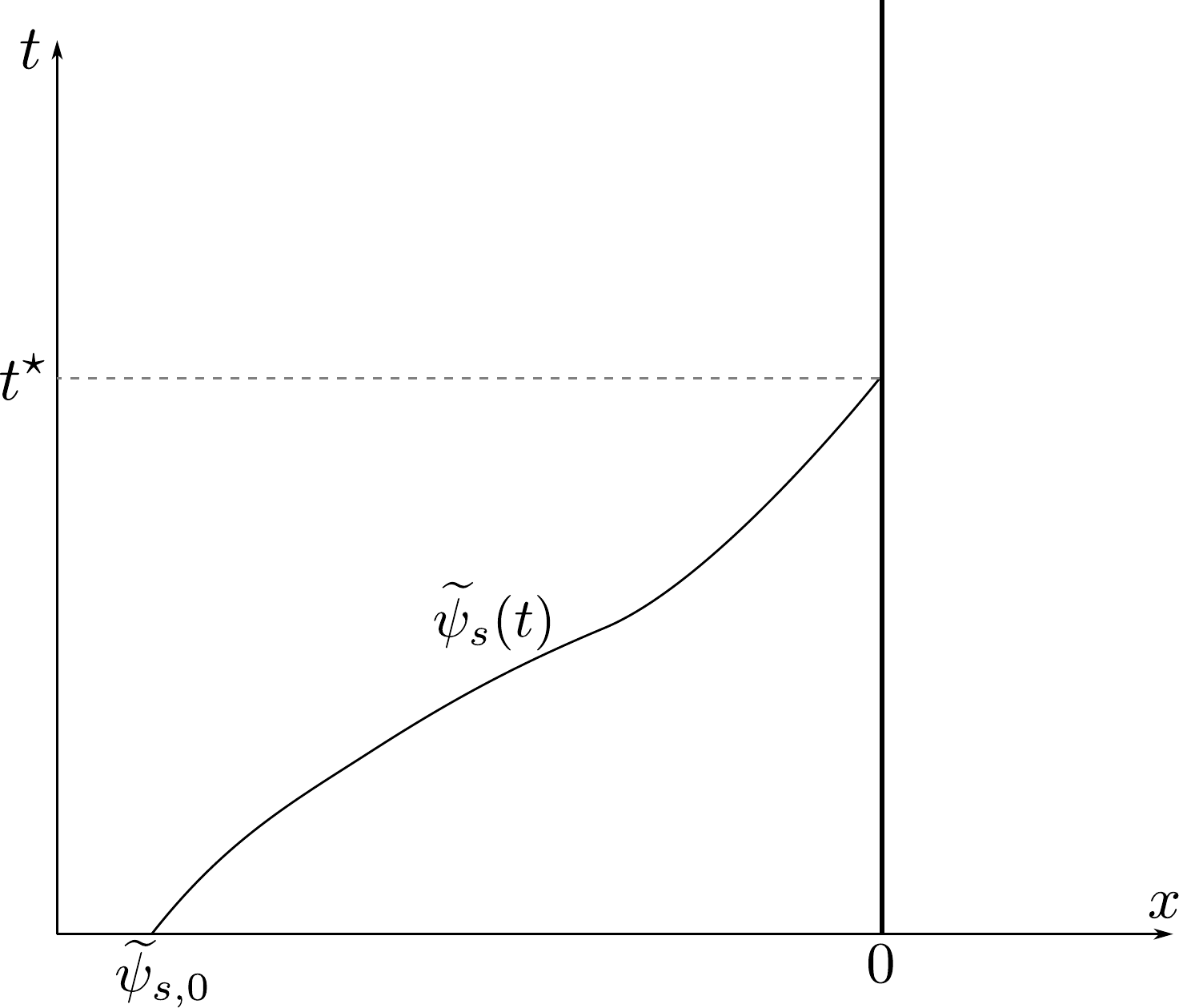}
}
\caption{Sketch of the shock paths.}
\label{fig:sketch_shocks}
\end{figure}

\begin{Theorem}\label{Th.shocks}
Let $f,g\in \mathcal{C}^2(\RR)$ and $(\sigma,\uu_-,\uu_+)\in\RR^3$ satisfying~\eqref{equilibrium-shock}-\eqref{RH-shock}-\eqref{gLax-shock}-\eqref{stab-spec-shock} and~\eqref{hyp-convex-shock}.
For any $C_0>1$, there exists $\epsilon>0$ and $C>0$ such that for any $\tpsi_{s,0}<0$ and $\psi_{0}\in\RR$ and any $\tv_0\in BUC^1(\RR^\star\setminus\{\tpsi_{s,0}\})$ satisfying
\begin{equation}\label{init-condition-shocks}
\begin{array}{rl}
\Norm{\tv_0}_{L^\infty(\RR^\star\setminus\{\tpsi_{s,0}\})}&\leq\epsilon\,,\\
\Norm{(\sign(f''(\uu_+))\,\d_x\tv_0)_-}_{L^{\infty}(\RR^+)}&\leq\epsilon\,,\\
\Norm{(\sign(f''(\uu_-))\,\d_x\tv_0)_-}_{L^{\infty}(\RR^-\setminus\{\tpsi_{s,0}\})}&\leq\epsilon\,,
\end{array}
\end{equation}
there exist a time $t^\star\in(0,+\infty)$,
\begin{itemize} 
\item a $\cC^0$ function $\psi:\,\RR_+\to\RR$ that is $\cC^2$ on $\RR_+\setminus\{t^\star\}$  and such that $\psi(0)=\psi_0$, 
\item a $\cC^2$ function $\tpsi_s:\,[0,t^\star]\to \RR_-$ such that $\tpsi_s$ is negative on $[0,t^\star)$, $\tpsi_s(0)=\tpsi_{s,0}$ and $\tpsi_s(t^\star)=0$,
\end{itemize}
so that, with $\psi_s=\psi_{|[0,t^\star]}+\tpsi_s$, the entropy solution to~\eqref{eq-u}, $u$, generated by the initial data $(\uU+\tv_0)(\cdot+\psi_{0})$ belongs to $BUC^1(\Omega_{\psi,\psi_s})$ and satisfies\footnote{In the first three inequalities we sacrifice consistency to the sake of concision and readilibility and write $\RR^-$ even when $0\leq t< t^\star$ and notational conventions used elsewhere would require $\RR^-\setminus\{\tpsi_s(t)\}$ or $\RR^-\setminus\{\tpsi_{s,0}\}$.} for any $t\geq 0$
\begin{align*}
\Norm{u(t,\cdot-\psi(t))-\uu_\pm}_{L^{\infty}(\RR^\pm)}&\leq \Norm{\tv_0}_{L^{\infty}(\RR^\pm)} C_0\,e^{g'(\uu_\pm)\,t}\,,\\
\Norm{(\sign(f''(\uu_\pm))\,\d_x u(t,\cdot-\psi(t)))_-}_{L^{\infty}(\RR^\pm)}&\leq \Norm{(\sign(f''(\uu_\pm))\,\d_x\tv_0)_-}_{L^{\infty}(\RR^\pm)} C_0\,e^{g'(\uu_\pm)\,t}\,,\\
\Norm{\d_x u(t,\cdot-\psi(t))}_{L^{\infty}(\RR^\pm)}&\leq \Norm{\d_x \tv_0}_{L^{\infty}(\RR^\pm)} C_0\,e^{g'(\uu_\pm)\,t}\,,\\
\abs{\psi_s'(t)-f'(\uu_-)}&\leq \Norm{\tv_0}_{L^{\infty}(\RR^-\setminus\{\tpsi_{s,0}\})}\,C\, e^{g'(\uu_-)\, t}\,,\qquad t\leq t^\star\,,\\
\abs{\psi'(t)-\sigma}&\leq \Norm{\tv_0}_{L^{\infty}(\RR^\star\setminus\{\tpsi_{s,0}\})}\,C\, e^{\max(\{g'(\uu_+),g'(\uu_-)\})\, t}\,,
\end{align*}
and moreover there exists $\psi_\infty$ such that 
\[
\abs{\psi_\infty-\psi_{0}}\,\leq \Norm{\tv_0}_{L^{\infty}(\RR^\star\setminus\{\tpsi_{s,0}\})} C\,,
\]
and for any $t\geq 0$
\[
\abs{\psi(t)-\psi_\infty-t\,\sigma}\,\leq \Norm{\tv_0}_{L^{\infty}(\RR^\star\setminus\{\tpsi_{s,0}\})} C\, e^{\max(\{g'(\uu_+),g'(\uu_-)\})\, t}\,.
\]
\end{Theorem}

\begin{proof}
Here again we follow the extension/patching strategy used for the previous results, assume without loss of generalty that $\psi_0=0$, and correspondingly drop some tildes. With a straightforward adaptation of Lemma~\ref{L.multivalued} and using Proposition~\ref{P.constant}, we find that for $\epsilon>0$ sufficiently small and for any $v_0\in BUC^1(\RR^\star\setminus\{\psi_{s,0}\})$ satisfying~\eqref{init-condition-shocks}, there exists
$u_l,u_c,u_r\in BUC^1(\RR^+\times\RR)$ global classical solutions to~\eqref{eq-u} with initial data such that
\[
\begin{cases}
u_l(0,x)=\uu_-+v_{0}&\text{ if $x<\psi_{s,0}$}\\
u_c(0,x)=\uu_-+v_{0}&\text{ if $x\in(\psi_{s,0},0)$}\\
u_r(0,x)=\uu_++v_{0}&\text{ if $x>0$}
\end{cases}
\]
and satisfying the desired estimates. 

We may now identify shock locations. Let $\psi_l$ and $\psi_r$ be defined by the differential equations
\[\psi_l'(t)\,=\, s_f(u_l(t,\psi_l(t)),u_c(t,\psi_l(t))) \text{ and } \psi_r'(t)\,=\, s_f(u_c(t,\psi_r(t)),u_r(t,\psi_r(t))) \]
with initial data $\psi_l(0)=\psi_{l,0}$ and $\psi_r(0)=0$. Then we observe that $\psi_l,\psi_r\in \cC^2(\RR^+)$ and 
\begin{align*}
\abs{\psi_l'(t)-f'(\uu_-)}&\leq \Norm{v_0}_{L^{\infty}(\RR^-\setminus\{\psi_{s,0}\})}\,C\, e^{g'(\uu_-)\, t}\,,\\
\abs{\psi_r'(t)-\sigma}&\leq \Norm{v_0}_{L^{\infty}(\RR^\star\setminus\{\psi_{s,0}\})}\,C\, e^{\max(\{g'(\uu_+),g'(\uu_-)\})\, t}\,.
\end{align*}
Since $f'(\uu_-)>\sigma$ and $\psi_{s,0}<0$ this implies that the time
\[t^\star=\argmin\ \{ t\in\RR^+ \ | \ \psi_l(t)=\psi_r(t)\}\]
is postive and finite. At last $\psi_f$ is defined by the differential equation
\[\psi_f'(t)\,=\, s_f(u_l(t,\psi_f(t)),u_r(t,\psi_f(t)))\]
with ``initial'' data $\psi(t^\star)=\psi_r(t^\star)$. Note that $\psi_f\in \cC^2(\RR^+)$ and that $\abs{\psi_f'(t)-\sigma}$ also decays exponentially with the same estimate as $\abs{\psi_r'(t)-\sigma}$. Then we set $\psi_s=(\psi_l)_{|[0,t^\star]}$ and
\[
\psi:\,\RR\to\RR\,,\qquad t\mapsto
\begin{cases}
\psi_r(t)&\text{ if $0\leq t\leq t^\star$}\\
\psi_f(t)&\text{ if $t>t^\star$}
\end{cases}\,.
\]
Again the last estimates on $\psi$ are obtained by integration with 
\[
\psi_\infty=\int_0^\infty (\psi'(t)-\sigma)\,\dd t\,.
\]

We can now construct the solution $u$. For any $t\in[0,t^\star]$, we define
\[
u(t,x)=
\begin{cases} 
u_l(t,x) & \text{ if } x<\psi_s(t)\\
u_c(t,x)  & \text{ if } \psi_s(t) < x<\psi(t)\\
u_r(t,x) & \text{ if } x>\psi(t)
\end{cases}\,.
\]
For subsequent times $t\in [t^\star,+\infty)$, we set
\[
u(t,x)=
\begin{cases} 
u_l(t,x) & \text{ if } x<\psi(t)\\
u_r(t,x) & \text{ if } x>\psi(t)
\end{cases}\,.
\]
One easily checks that the function $u$ is an entropy solution as soon as $\epsilon$ is sufficiently small, following the proof of Proposition~\ref{P.small-shock} (along the path $\{(t,\psi_s(t))\,|\,0\leq t\leq t^\star\}$) and Proposition~\ref{Th.shock} (along the path $\{(t,\psi(t))\,|\,t\geq 0\}$).
\end{proof}

\section{Transverse stability in the multidimensional framework}\label{s:multiD}

In the present section we consider some generalizations of the main results proved so far to multidimensional settings. In particular we now replace \eqref{eq-u} with
\begin{equation}\label{eq-u-multiD}
\d_t u+\Div \big(f(u)\big)=g(u)
\end{equation}
where the spatial variable $x$ belongs to $\RR^d$, $d\in\NN$, and $f:\RR\to\RR^d$.

Most of adaptations are rather straightforward and we only sketch main variations required in the process. We aim not at gaining new insights on the general stability problem but at demonstrating a certain robustness of the one-dimensional arguments.

Starting from spatial dimension $2$ the range of possible geometries for discontinuities becomes too wide to be reasonably covered here, even if one restricts to a few typical cases. Therefore, in the multidimensional case we only consider shockless perturbations.

\subsection{Asymptotic stability of constant states}

Propositions~\ref{P.constant-classical} and~\ref{P.high-order} translate almost verbatim to the multidimensional case.

\begin{Proposition}\label{P.constant-multiD}
Let $f\in\cC^2(\RR;\RR^d)$, $g\in \cC^2(\RR)$ and $\uu\in\RR$ be such that 
\[
g(\uu)=0\qquad\textrm{and}\qquad g'(\uu)<0\,.
\]
Then for any $C_0>1$, there exists $\epsilon>0$ such that for any $v_0\in BUC^1(\RR^d)$ satisfying 
\[\Norm{v_0}_{W^{1,\infty}(\RR^d)}\leq \epsilon\,,\]
the initial data $u\id{t=0}=\uu+v_0$ generates a global unique classical solution to~\eqref{eq-u-multiD}, $u\in BUC^1(\RR^+\times\RR^d)$, and it satisfies for any $t\geq0$
\begin{align*}
\Norm{u(t,\cdot)-\uu}_{L^{\infty}(\RR^d)}&\leq \Norm{v_0}_{L^{\infty}(\RR^d)}C_0\,e^{g'(\uu)\,t}\ ;\\
\Norm{\nabla u(t,\cdot)}_{L^{\infty}(\RR^d;\RR^d)}&\leq \Norm{\nabla v_0}_{L^{\infty}(\RR^d;\RR^d)} C_0\,e^{g'(\uu)\,t}\,.
\end{align*}
\end{Proposition}

In the foregoing statement and henceforth, we measure vectors in $\RR^d$ with the Euclidean norm.

\begin{Proposition}\label{P.high-order-multiD} Under the assumptions of Proposition~\ref{P.constant-multiD}, if one assumes additionally that $f\in\cC^{k+1}(\RR;\RR^d)$, $g\in \cC^k(\RR)$ with $k\in\NN$, $k\geq 2$ then there exists $C_k>0$, depending on $f$, $g$ and $k$ but not on the initial data $v_0$, such that if $v_0\in BUC^k(\RR^d)$ additionally to constraints in Proposition~\ref{P.constant-multiD}, then the global unique classical solution to~\eqref{eq-u-multiD} emerging from the initial data $\uu+v_0$ satisfies
$u\in BUC^{k}(\RR^+\times\RR^d)$ and for any $\alpha$ multi-index of size $|\alpha|=k$ and $t\geq0$
\[
\Norm{\d^\alpha u(t,\cdot)}_{L^{\infty}(\RR^d)}\leq 
\Norm{\d^\alpha v_0}_{L^{\infty}(\RR^d)}e^{C_k\,\|v_0\|_{W^{1,\infty}}\,(1+\|v_0\|_{W^{1,\infty}}^{k-1})}\,e^{g'(\uu)\,t}\,.
\]
\end{Proposition}

Proofs are also essentially identical. The proof involves now the consideration of operators $L_{a,\,b}=-a\cdot\nabla+b$, but still with $a$ close to $f'(\uu)$ and $b$ close to $g'(\uu)$. Since transport operators are not elliptic when $d>1$, their domains --- the set of $v\in BUC^0(\RR^d)$ such that $L_{a,\,b}v\in BUC^0(\RR^d)$ --- cannot be identified with a classical function space. Note however that this does not alter any part of the argument and that in particular one may still apply~\cite[Chapter~5, Theorem~3.1]{Pazy}  with $X=BUC^0(\RR^d)$ and $Y=BUC^1(\RR^d)$ even though $Y$ is now a common core, rather than a common domain.

Key estimates are provided in the following lemma.

\begin{Lemma}\label{l:resolvent-multiD}
Assume $a\in BUC^1(\RR^d;\RR^d)$, $b\in BUC^0(\RR^d)$.\\
{\bf (i).} Then for any $\lambda\in\CC$ such that 
\[
\Re(\lambda)>\sup_{\RR^d} b(\cdot)\,,\] 
for any $F\in BUC^0(\RR^d)$, there exists a unique $\check{v}(\,\cdot\,;\lambda)\in BUC^0(\RR^d)$ such that 
\[
(\lambda-L_{a,\,b})\,\check{v}(\,\cdot\,;\lambda)\,=\,F
\]
and moreover
\[
\Norm{\check v(\,\cdot\,;\lambda)}_{L^\infty(\RR^d)}\leq \frac1{\Re\lambda-\sup_{\RR^d} b(\cdot)}\Norm{F}_{L^\infty(\RR^d)}\,.
\]
{\bf (ii).} Assume moreover that 
\[
b\textrm{ is constant}\qquad\textrm{ and }\qquad\Re(\lambda)>b+\sup_{\RR^d}\|\dd a\|(\cdot)
\]
then for any $F\in BUC^1(\RR^d)$, $\check{v}(\,\cdot\,;\lambda)\in BUC^1(\RR^d)$ and
\[
\Norm{\nabla\check v(\,\cdot\,;\lambda)}_{L^\infty(\RR^d)}\leq \frac1{\Re\lambda-b-\sup_{\RR^d}\|\dd a\|(\cdot)}\Norm{\nabla F}_{L^\infty(\RR^d)}\,.
\]
\end{Lemma}

\begin{proof}
One may proceed as in the one-dimensional case, with generalized formula
\[
\check v(x;\lambda)\eqdef \int_{-\infty}^0 e^{\int_s^0(b(X(\sigma;x))-\lambda)\,\dd \sigma}F(X(s;x))\,\dd s
\]
where $X(\cdot,x)$ is such that $X(0,x)=x$ and $\forall s\in\RR$, $\d_sX(s;x)=a(X(s;x))$. Since $a$ is Lipschitz, $X(\cdot,x)$ is indeed globally well-defined. When moreover $b$ is constant one derives from the latter
\begin{align*}
\dd_x\check v(x;\lambda)(h)
&=\int_{-\infty}^0 e^{\int_s^0(b-\lambda)\,\dd \sigma}\dd F(X(s;x))(\dd_x X(s;x)(h))\,\dd s
\end{align*}
and the claim follows from the observation that for any $s\leq0$
\[
\|\dd_x X(s;x)(h)\|
\leq e^{\int_s^0\|\dd a(X(\sigma;x))\|\,\dd \sigma}\|h\|
\leq e^{\int_s^0\|\dd a\|_{L^\infty}\,\dd \sigma}\|h\|\,.
\]
\end{proof}

\subsection{Asymptotic stability of Riemann shocks}

We now turn to the stability of plane Riemann shocks. For the sake of clarity and without loss of generality, we fix the direction of propagation of the reference Riemann shock, split spatial variables accordingly $x=(\xi,y)\in\RR\times\RR^{d-1}$, and correspondingly $f=(\fpar,\fperp)$.

We consider a plane wave $\uu$, 
\[\uu(t,x)=\uU(\xi-(\psi_0+\sigma t))\,,\] 
with $\psi_0\in\RR$, $\sigma\in\RR$ and $\uU$ such that
\begin{equation}\label{uU-multiD}\uU(\xi)=\begin{cases}
\uu_- &\text{ if }  \xi<0\\
\uu_+ &\text{ if }  \xi>0
\end{cases}\end{equation}
where $(\uu_-,\uu_+)\in\RR^2$, $\uu_+\neq\uu_-$ are such that
\begin{equation}\label{shock-multiD}
g(\uu_+)=0\,,\quad g(\uu_-)=0\quad \text{ and } \quad
\fpar(\uu_+)-\fpar(\uu_-)=\sigma(\uu_+-\uu_-)
\end{equation}
and we assume that $\uu$ is strictly entropy-admissible, that is 
\begin{equation}\label{gKru-shock}
\begin{cases}
\qquad\qquad\qquad\sigma\,>\,\fpar'(\uu_+)\,,&\\[0.5em]
\frac{\fpar(\tau\,\uu_-+(1-\tau)\,\uu_+)-\fpar(\uu_-)}{\tau\,\uu_-+(1-\tau)\,\uu_+-\uu_-}>
\frac{\fpar(\tau\,\uu_-+(1-\tau)\,\uu_+)-\fpar(\uu_+)}{\tau\,\uu_-+(1-\tau)\,\uu_+-\uu_+}&\qquad\textrm{for any }\ \tau\in(0,1)\,,\\[0.5em]
\qquad\qquad\qquad \fpar'(\uu_-)\,>\,\sigma\,,
\end{cases}
\end{equation}
(though this does not play a strong role in our analysis; see Remark~\ref{R.Lax-vs-Oleinik}). Moreover we require
\begin{equation}\label{stab-multiD}
g'(\uu_+)<0 \quad \text{ and } \quad g'(\uu_-)<0\,.
\end{equation}

To accurately account for discontinuities, we introduce
\begin{align*}
\Omega^{+}_{\psi_0}
&\eqdef\{\,(\xi,y)\in \RR\times\RR^{d-1}\,;\quad\xi>\psi_0(y)\,\}\\
\Omega^{-}_{\psi_0}
&\eqdef\{\,(\xi,y)\in \RR\times\RR^{d-1}\,;\quad\xi<\psi_0(y)\,\}\\
\Omega_{+}^{\psi}
&\eqdef\{\,(t,\xi,y)\in \RR_+\times\RR\times\RR^{d-1}\,;\quad\xi>\psi(t,y)\,\}\\
\Omega_{-}^{\psi}
&\eqdef\{\,(t,\xi,y)\in \RR_+\times\RR\times\RR^{d-1}\,;\quad\xi<\psi(t,y)\,\}\\
\Omega_{\psi_0}
&\eqdef\Omega^{+}_{\psi_0}\cup\Omega^{-}_{\psi_0}\,,\qquad
\Omega^{\psi}
\eqdef\Omega_{+}^{\psi}\cup\Omega_{-}^{\psi}\,.
\end{align*}

\begin{Theorem}\label{Th.shock-multiD}
Let $f\in \mathcal{C}^2(\RR;\RR^d)$, $g\in \mathcal{C}^2(\RR)$ and $(\sigma,\uu_-,\uu_+)\in\RR^3$ satisfying~\eqref{shock-multiD}-\eqref{gKru-shock}-\eqref{stab-multiD}. For any $C_0>1$, there exists $\epsilon>0$ and $C>0$ such that for any $\psi_0\in BUC^1(\RR^{d-1})$ and $v_0\in BUC^1(\Omega_{\psi_0})$ satisfying
\begin{equation*}
\begin{array}{rl}
\Norm{v_0}_{W^{1,\infty}(\Omega_{\psi_0})}&\leq\epsilon\,,\\
\Norm{\nabla_y\psi_0}_{L^{\infty}(\RR^{d-1})}&\leq\epsilon\,,
\end{array}
\end{equation*}
there exists $\psi\in\cC^2(\RR^+\times\RR^{d-1})$ with initial data $\psi(0,\cdot)=\psi_0$ such that the entropy solution to~\eqref{eq-u-multiD}, $u$, generated by the initial data 
\[
u(0,\xi,y)=\uU(\xi-\psi_0(y))+v_0(\xi,y)\,,
\]
$\uU$ being as in \eqref{uU-multiD}, belongs to $ BUC^1(\Omega^\psi)$ and satisfies for any $t\geq 0$ 
\begin{align*}
\Norm{u(t,\cdot,\cdot)-\uu_\pm}_{L^{\infty}(\Omega^{\pm}_{\psi(t,\cdot)})}&\leq \Norm{v_0}_{L^{\infty}(\Omega^{\pm}_{\psi_0})} C_0\,e^{g'(\uu_\pm)\,t}\,,\\
\Norm{\nabla u(t,\cdot,\cdot)}_{L^{\infty}(\Omega^{\pm}_{\psi(t,\cdot)};\RR^d)}&\leq \Norm{\nabla v_0}_{L^{\infty}(\Omega^{\pm}_{\psi_0};\RR^d)} C_0\,e^{g'(\uu_\pm)\,t}\,,\\
\Norm{\nabla_y\psi(t,\cdot)}_{L^\infty(\RR^{d-1};\RR^{d-1})}&\leq 
\Norm{\nabla_y\psi_0}_{L^\infty(\RR^{d-1};\RR^{d-1})}\,C_0\,
\,+\,\Norm{ v_0}_{L^{\infty}(\Omega_{\psi_0})}\,C\\
\Norm{\d_t\psi(t,\cdot)+s_{\fperp}(\uu_-,\uu_+)\cdot\nabla_y\psi(t,\cdot)-\sigma}_{L^\infty(\RR^{d-1})}&\leq \Norm{v_0}_{L^{\infty}(\Omega_{\psi_0})}
\,C\, e^{\max(\{g'(\uu_+),g'(\uu_-)\})\, t}\,,
\end{align*}
where $s_{\fperp}$ is as defined as in \eqref{slope-function}, and moreover there exists $\psi_\infty$ such that 
\[
\Norm{\psi_\infty-\psi_0}_{L^\infty(\RR^{d-1})}\,\leq \Norm{v_0}_{L^{\infty}(\Omega_{\psi_0})} C\,,
\]
and for any $t\geq 0$
\[
\Norm{\psi(t,\cdot)-\psi_\infty(\cdot-t\,s_{\fperp}(\uu_-,\uu_+))-t\,\sigma}_{L^\infty(\RR^{d-1})}\,\leq \Norm{v_0}_{L^{\infty}(\Omega_{\psi_0})} C\, e^{\max(\{g'(\uu_+),g'(\uu_-)\})\, t}\,.
\]
\end{Theorem}

\begin{proof}
The proof follows the same scheme as the one of the proof of Theorem~\ref{Th.shock}. To start with, even though one cannot reduce the proof to the case when $\psi_0=0$, the Lipschitz assumption on $\psi_0$ is sufficient to extend both smooth parts of $v_0$ so as to apply Proposition~\ref{P.constant-multiD} to both extensions and receive functions $u_-$ and $u_+$. Once the shock position is determined --- through the equation $\xi=\psi(t,y)$ ---, that is, once $\psi$ is known, the sought solution is obtained by gluing $u_-$ and $u_+$.

Once again the evolution of $\psi$ is determined from the Rankine-Hugoniot conditions. They read
\[
\d_t\psi(t,\cdot)+s_{\fperp}(u_-(t,\psi(t,\cdot),\cdot),u_+(t,\psi(t,\cdot),\cdot))\cdot\nabla_y\psi(t,\cdot)
\,=\,s_{\fpar}(u_-(t,\psi(t,\cdot),\cdot),u_+(t,\psi(t,\cdot),\cdot))\,.
\]
At the level of regularity considered here, the local-well posedness of the foregoing PDE is classical and we may again focus on proving bounds. Uniform bounds on $\psi(t,\cdot)-\sigma t$ and $\nabla_y \psi$, hence also the global well-posedness, is derived by arguing as in the proof of Proposition~\ref{P.constant-multiD}. Indeed, Lemma~\ref{l:resolvent-multiD} yields the existence of and bounds for the evolution system generated by the time-dependent linear operator $L_{a(t,\cdot),\,0}=-a(t,\cdot)\cdot\nabla$, where $a\in\cC^0([0,T),BUC^1(\RR^{d-1}))$. We may then apply Duhamel's formula on the equation
\begin{multline*}
\d_t(\psi(t,\cdot)-\sigma t)+s_{\fperp}(u_-(t,\psi(t,\cdot),\cdot),u_+(t,\psi(t,\cdot),\cdot))\cdot\nabla_y(\psi(t,\cdot)-\sigma t) \\
\,=\,s_{\fpar}(u_-(t,\psi(t,\cdot),\cdot),u_+(t,\psi(t,\cdot),\cdot))-s_{\fpar}(\uu_-,\uu_+)\,,
\end{multline*}
and use that  $ u_--\uu_-$ and $ u_+-\uu_+$ decay exponentially in time in $W^{1,\infty}$.
Then the decay bounds follow from substracting $\left(s_{\fperp}(u_-(t,\psi(t,\cdot),\cdot),u_+(t,\psi(t,\cdot),\cdot))-s_{\fperp}(\uu_-,\uu_+)\right)\cdot\nabla_y\psi(t,\cdot)$ to each side of the above equation and using again the exponential decay in time in $L^\infty$ of $u_--\uu_-$ and $u_+-\uu_+$, with 
\begin{align*}
\psi_\infty(y)
&\,\eqdef\,\psi_0(y)
+\int_0^\infty\frac{\dd}{\dd t}\left(\psi(\cdot,y+s_{\fperp}(\uu_-,\uu_+)\,\cdot)-\sigma\,\cdot\right)(t)\,\dd t\\
&\,=\,\psi_0(y)
+\int_0^\infty\left(\d_t\psi
+s_{\fperp}(\uu_-,\uu_+)\cdot\nabla_y\psi\right)(t,y+s_{\fperp}(\uu_-,\uu_+)t)-\sigma\,\dd t
\,.
\end{align*}
Finally, the $L^\infty$ bounds on $u_\pm$ allow to guarantee, lessening $\epsilon$ if necessary, that the strict entropy-admissibility of the shock propagates for positive times, and the uniqueness of entropy solutions is provided by the theory of Kru\v zkov~\cite{Kruzhkov}.
\end{proof}

The proof of Theorem~\ref{P.shock-higher-derivatives} also carries over to the multidimensional case. Note in particular that the Lipschitz condition on $\psi_0$ is already sufficient to extend smooth parts of $v_0$ while maintaining higher-order regularity.

\begin{Proposition}\label{P.shock-higher-derivatives-multiD}
Let $k\in\NN$, $k\geq2$, $f\in\cC^{k+1}(\RR;\RR^d),\ g\in \cC^k(\RR)$ and $(\sigma,\uu_-,\uu_+)\in\RR^3$ satisfying~\eqref{shock-multiD}-\eqref{stab-multiD} and defining a strictly entropic plane wave. There exists $\epsilon>0$ and $C_k$ such that for any $\psi_0\in BUC^1(\RR^d)$ and $v_0\in BUC^k(\Omega_{\psi_0})$ satisfying
\begin{equation*}
\Norm{v_0}_{W^{1,\infty}(\Omega_{\psi_0})}\leq\epsilon \quad  \text{ and } \quad 
\Norm{\nabla_y\psi_0}_{L^{\infty}(\RR^d)}\leq\epsilon\,,
\end{equation*}
there exist $\psi\in\cC^{2}(\RR^+)$ and $u\in BUC^k(\Omega^\psi)$ with initial data $\psi(0,\cdot)=\psi_0$ and $u(0,\cdot,\cdot)$ given by 
\[
u(0,\xi,y)=\uU(\xi-\psi_0(y))+v_0(\xi,y)\,,
\]
$\uU$ being as in \eqref{uU-multiD}, such that $u$ is an entropy solution to~\eqref{eq-u-multiD} and satisfies for any $t\geq 0$ and any $j\in \{0,\dots,k\}$
\begin{align*}
\Norm{u(t,\cdot,\cdot)-\uu_\pm}_{W^{j,\infty}(\Omega_{\psi(t,\cdot)}^\pm)}&\leq \Norm{v_0}_{W^{j,\infty}(\Omega_{\psi_0}^\pm)} C_k\,e^{g'(\uu_\pm)\,t}\,.
\end{align*}
If moreover $\psi_0\in BUC^{k+1}(\RR^d)$, then $\psi\in BUC^{k+1}(\RR\times\RR^d)$ and for any $t\geq 0$ and any $j\in \{1,\dots,k\}$
\begin{align*}
\Norm{
(\d_t+s_{\fperp}(\uu_-,\uu_+)\cdot\nabla_y)^{j+1}\psi(t,\cdot)}_{L^\infty(\RR^d)}&\leq \Norm{v_0}_{W^{j,\infty}(\Omega_{\psi_0})}\,C_k\, e^{\max(\{g'(\uu_+),g'(\uu_-)\})\, t}\,.
\end{align*}
\end{Proposition}

\section{Further comments}\label{s:conclusion}

Finally we conclude our contribution by replacing the present analysis in the more general framework of stability of discontinuous solutions of hyperbolic systems of balance laws. In particular, we comment on which feature of Riemann shocks of scalar laws is involved in each of the derived properties.

To begin with, note that this is from the scalar character of the equations that stems the fact that the analysis of piecewise smooth solutions may be reduced to on one hand the analysis of smooth solutions and on the other hand of how to glue them according to the Rankine-Hugoniot conditions. The strategy may also be used in multidimensional cases (as in Section~\ref{s:multiD}) or with wave profiles that are not piecewise constant (as in~\cite{DR2}). The argument is crucially used to allow for perturbations adding discontinuities. However, as long as perturbations do not change the regularity structure, up to some loss in optimality one may replace the argument with a direct stability analysis whose starting point is expounded below. In the scalar case, the latter restriction amounts essentially to requiring smoothness on perturbations, whereas in the system case it should also include compatibility conditions at discontinuities, for instance met by perturbations supported far away from the original discontinuities\footnote{We recall that during the finalization of the present contribution we have been informed that a system case has been analyzed in~\cite{YZ} and we refer to it for such an example of analysis.}. Unfortunately for the moment the consideration of discontinuous perturbations in a system case still lies beyond reach of both frameworks.

To give an idea of what a direct stability analysis for Riemann shocks would look like, we now write the corresponding spectral stability problem (in the one-dimensional case). The question is whether one may solve uniquely in $(\check{v}(\,\cdot\,;\lambda),\check{\psi}(\lambda))$ and continuously with respect to any $(F,\varphi)$ the system
\begin{align*}
(\lambda&+(f'(\uu_+)-\sigma)\d_x-g'(\uu_+))\,\check{v}(\,\cdot\,;\lambda)\,=\,F\qquad\textrm{on }\RR_+^\star\,,\\
(\lambda&+(f'(\uu_-)-\sigma)\d_x-g'(\uu_-))\,\check{v}(\,\cdot\,;\lambda)\,=\,F\qquad\textrm{on }\RR_-^\star\,,\\
\lambda\check{\psi}(\lambda)
&-\,\left(\frac{f'(\uu_+)-\sigma}{\uu_+-\uu_-}\check{v}(0^+;\lambda)
-\frac{f'(\uu_-)-\sigma}{\uu_+-\uu_-}\check{v}(0^-;\lambda)\right)\,=\,\varphi\,.
\end{align*}
An inspection of the proof of Lemma~\ref{l:resolvent} shows that since $f'(\uu_+)<\sigma$, when $\Re(\lambda)>g'(\uu_+)$, $\check{v}(\,\cdot\,;\lambda)$ is uniquely and continuously determined on $\RR_+^\star$ by the first equation, and obtained there by solving from $+\infty$. Likewise, when $\Re(\lambda)>g'(\uu_-)$, $\check{v}(\,\cdot\,;\lambda)$ is uniquely and continuously determined on $\RR_-^\star$ by the second equation, and obtained there by solving from $-\infty$. Thus when $\Re(\lambda)>\max(\{g'(\uu_-);g'(\uu_+)\})$, the problem is uniquely solved if and only if $\lambda\neq 0$. A closer examination reveals that it follows from $f'(\uu_+)<\sigma<f'(\uu_-)$ that the spectrum is 
\[
\left\{\ \lambda\ ;\ \Re(\lambda)\leq \max(\{g'(\uu_-);g'(\uu_+)\})\ \right\}\cup\{0\}
\]
and that when $\max(\{g'(\uu_-);g'(\uu_+)\})<0$, $0$ has multiplicity $1$ (in the sense provided by resolvent singularities). The consideration of the Riemann shocks may leave the reader with the deceptive impression that in scalar balance laws the stability of traveling waves is determined by smooth parts of the profiles whereas shocks play a passive role, encoding invariance by space translation. This special feature is however due to the fact that $g(\uU)$ is continuous across discontinuities. In contrast, in~\cite{DR2} we provide examples of some cases when positions of discontinuities are exponentially stable and of some others when they bring exponential instabilities even when the smooth parts are point-wise dissipative (in a weighted sense that yields equivalent norms).

In the foregoing paragraph and all through the text, not much care was required in the sense in which spectral properties were considered. Indeed, any reasonable functional framework is associated with the same spectrum. This is also reflected at the nonlinear level in the fact that stability is also obtained in arbitrarily high regularity. This is easy to see at the spectral level when profiles are piecewise constant, but it is more fundamentally related to the ellipticity of underlying operators, hence here to the fact that $f'(\uU)-\sigma$ is bounded away from zero. When this condition is relaxed (to include sonic points), one may still deduce, at any level of regularity, nonlinear stability from spectral stability. Yet spectral properties do depend on functional choices, and specifically on the degree of regularity under consideration. In particular in~\cite{DR2} we exhibit an instability mechanism that is seen in the topology of $BUC^k$, $k\geq 1$, but not in $BUC^0$, which demonstrates that not all key properties in scalar balance laws are captured by $L^\infty$-based estimates.

The present paper opens at least two clear close follow-up : the stability of more general traveling waves of scalar balance laws, and  the stability of Riemann shocks in general systems. The former is studied in details in the companion paper~\cite{DR2} where breakthrough obtained here are combined with an elucidation of the effects of the presence of sonic points in wave profiles so as to derive an essentially complete picture for balance laws. The latter is left for future work.

\newcommand{\etalchar}[1]{$^{#1}$}
\newcommand{\SortNoop}[1]{}

\end{document}